\documentclass[10pt,a4,leqno]{amsart}
\usepackage{amsthm}
\usepackage{amsmath}
\usepackage{amssymb}
\usepackage{graphicx}
\usepackage{multicol}
\usepackage{bbm}
\usepackage[
colorlinks=true, citecolor=blue, linkcolor=blue, urlcolor=red]{hyperref}
\usepackage{color}
\usepackage[all]{xy}
\allowdisplaybreaks
\usepackage{mathrsfs}
\usepackage{tikz}
\usetikzlibrary{cd}
\usepackage[all]{xy}
\usepackage{here}

\makeatletter

\@addtoreset{equation}{section}
\makeatother

\makeatletter
\@namedef{subjclassname@2020}{%
  \textup{2020} Mathematics Subject Classification}
\makeatother

\newtheoremstyle{my theoremstyle}
{1.0em}                    
    {1.0em}                    
    {\itshape}                   
    {}                           
    {\scshape}                   
    {.}                          
    {.5em}                       
    {}  

\newtheoremstyle{dfn}
{1.0em}                    
    {1.0em}                    
    {}                   
    {}                           
    {\scshape}                   
    {.}                          
    {.5em}                       
    {}  
    
\theoremstyle{my theoremstyle}
   \newtheorem{thm}{Theorem}[section]
   \newtheorem{lem}[thm]{Lemma}
   \newtheorem{prop}[thm]{Proposition}
   \newtheorem{cor}[thm]{Corollary}
   
\theoremstyle{dfn}
   \newtheorem{dfn}[thm]{Definition}
   
\theoremstyle{remark}   
   
   \newtheorem{rmk}[thm]{{\scshape Remark}}

\renewcommand{\labelenumi}{(\roman{enumi})}

\newcommand{\C}{\mathbb{C}}
\newcommand{\Q}{\mathbb{Q}}
\newcommand{\Z}{\mathbb{Z}}

\newcommand{\R}{\mathbb{R}}

\renewcommand{\P}{\mathbb{P}}

\renewcommand{\d}{\delta}

\renewcommand{\l}{\lambda}
\def\L{{\mathscr{L}}}

\newcommand{\spec}{\operatorname{Spec}}

\newcommand{\rk}{\operatorname{rank}}

\newcommand{\ch}{\operatorname{CH}}

\newcommand{\del}{\partial}
\newcommand{\Aut}{{\mathrm{Aut}}}
\def\aru{{\arrow[u,"\rotatebox{90}{$\in$}",phantom]}}

\numberwithin{equation}{section}

\date{\today}

\begin{document}
\title[Construction of Higher Chow cycles on cyclic coverings of $\P^1 \times \P^1$]{Construction of Higher Chow cycles on cyclic coverings of $\P^1 \times \P^1$}
\author{Yusuke Nemoto and Ken Sato}
\date{\today}
\address{Graduate School of Science and Engineering, Chiba University, 
Yayoicho 1-33, Inage, Chiba, 263-8522 Japan.}
\email{y-nemoto@waseda.jp}

\address{Department of Mathematics, Institute of Science Tokyo, 2-12-1 Ookayama, Meguro-ku, Tokyo 152-8551, Japan.}
\email{sato.k.da@m.titech.ac.jp}


\begin{abstract}
In this paper, we construct higher Chow cycles of type $(2, 1)$ on a certain family of surfaces, which are constructed by a product of certain hypergeometric curves of degree $N$.  
We prove that for a very general member, these cycles are linearly independent over $\Z$ and generate a subgroup of $\operatorname{rank} \ge 36 \cdot \varphi(N)$, where $\varphi(N)$ is Euler's totient function, by computing the image of the transcendental regulator map. 
\end{abstract}

\maketitle

\section{Introduction}
For an algebraic variety $X$, Bloch \cite{Bloch} introduced \textit{higher Chow groups} $\ch^p(X,q)$ for a variety.
These groups are generalizations of the classical Chow groups since we have an isomorphism $\ch^p(X, 0) \simeq \ch^p(X)$.
When $X$ is smooth, $\ch^p(X,q)$ coincides with the motivic cohomology $H_{\mathcal{M}}^{2p-q}(X, \Z(p))$, and is related to many aspects of algebraic geometry, $K$-theory, and number theory.
Higher Chow groups of codimension $p=1$ are described by classical invariants, e.g., $\ch^1(X,0)=\mathrm{Pic}(X)$, $\ch^1(X,1)=\Gamma(X,\mathscr{O}_X^\times)$.
However, their structures are mysterious when codimension $p>1$.

This article studies $\ch^2(X,1)$ for a smooth projective surface $X$.
An essential part of $\ch^2(X,1)$ is the {\it indecomposble part} $\ch^2(X, 1)_{\mathrm{ind}}$, which is the cokernel of the map $\C^\times\otimes \mathrm{Pic}(X) \to \ch^2(X,1)$ induced by the intersection product.
Construction of non-trivial elements in $\ch^2(X, 1)_{\mathrm{ind}}$ is done for mainly $K3$ surfaces and abelian surfaces (e.g., \cite{MS}, \cite{Spi}, \cite{Sre14}, \cite{CDKL}). 
On the contrary, few construction is known results on $\ch^2(X, 1)_{\mathrm{ind}}$ for a surface $X$ of general type except \cite{MSC}.

In this paper, we focus on a quotient of a product of certain hypergeometric curves, which are the smooth models of the projective closure of the affine curves defined by
\begin{equation*}
y^N = x^a(1-x)^b(1-\lambda x)^c\quad (\l\in\P^1\setminus\{0,1,\infty\}). 
\end{equation*}
Here, $a$, $b$, $c$, and $N$ are positive integers such that $\gcd(a, b, c, N)=1$. 
Note that hypergeometric curves have a natural $\mu_N$-action on $y$, where $\mu_N$ is the group of the $N$th root of unity.

For integers $N\ge 2$ and $A$ satisfying $0<A<N$ and $\gcd(N,A)=1$, let $C_1\to \P^1\setminus\{0,1,\infty\}$ (resp. $C_2\to \P^1\setminus\{0,1,\infty\}$) be the family of hypergeometric curves whose parameters satisfy
\begin{equation}\label{condparameter}
a=b=c=A\quad (\text{resp. } a=b=c=N-A).
\end{equation}
We use $\l_i$ for the indefinite of the parameter space of $C_i$.
We define the family of surfaces $S$ by the minimal resolution of the quotient of $C_1\times C_2$ by the diagonal $\mu_N$-action.
By taking a suitable \'etale base change by $T\to (\P^1\setminus\{0,1,\infty\})^2$, 
we have a family of surfaces $S \to T$.  
We explicitly construct a collection of families of higher Chow cycles which generate a subgroup $\widetilde{\Xi}_t$ of $\ch^2(S_t, 1)_{\mathrm{ind}}$ for each $t\in T$, where $S_t$ denotes the fiber over $t$.  
Then we have the following.

\begin{thm}[Theorem \ref{mainthm}] \label{mainintro}
Suppose $N \neq 2$ and $(N+1)/3 \leq A  \leq (2N-1)/3$. 
Then for very general $t\in T$, we have 
$$\rk \widetilde{\Xi}_t\ge 36\cdot \varphi(N),$$
where $\varphi$ is Euler's totient function.
In particular, we have 
$$
\rk \ch^2(S_t, 1)_{\rm ind} \geq 36 \cdot \varphi(N)
$$ 
for such a surface.
\end{thm}

By this theorem, the rank of $\ch^2(X_,1)_{\mathrm{ind}}$ for smooth complex projective surfaces $X$ can be arbitrary large.  
A similar result was essentially obtained by Terasoma \cite{Terasoma}, but our construction is totally different.

The key point of our construction is to regard $S_t$ as desingularized coverings of $\mathbb{P}^1 \times \mathbb{P}^1$ of degree $N$. 
We construct a higher Chow cycle $\xi_1^{(0)}-\xi_0^{(0)}\in \ch^2(S_t,1)$  using the exceptional curves and the pull-backs of the diagonal cycles on $\mathbb{P}^1 \times \mathbb{P}^1$, which is a rational curve by the condition \eqref{condparameter}.  
This cyclic covering construction is a generalization of the one used in \cite{Sato}.

First, we prove that the higher Chow cycle $\xi_1^{(0)}-\xi_0^{(0)}\in \ch^2(S_t,1)$ is a non-trivial indecomposable cycle.   
To prove this, we use the {\it transcendental regulator map} (see Definition \ref{transreg})
\begin{equation*}
r\colon \mathrm{CH}^2(S_t,1) \to \frac{F^1H^2(S_t,\C)^\vee}{H_2(S_t,\Q)} \twoheadrightarrow \frac{H^{2,0}(S_t)^\vee}{H_2(S_t,\Q)}, 
\end{equation*}
where the first map is the Beilinson regulator map and the second map is the natural projection. 
The transcendental regulator map factors through $\ch^2(S_t, 1)_{\mathrm{ind}}$, hence it suffices to show that the image of this map is non-trivial.

We evaluate the image of $(\xi_1^{(0)})_t - (\xi_0^{(0)})_t$ under the map $r$ at a certain relative holomorphic 2-form\footnote{
This 2-form is obtained by the wedge product of $dx/y$ on families of hypergeometric curves $C_i$, where we use the assumption $(N+1)/3 \leq A \leq (2N-1)/3$ to guarantee this form become holomorphic.} $\omega$, then we obtain a multi-valued holomorphic function $\mathcal{L}(t)$ on $T$ satisfying the following system of differential equations (see Theorem \ref{main:1}): 
\begin{equation}\label{PFdiff}
\left\{
\begin{aligned}
& \l_1(1-\l_1) \dfrac{\partial^2\mathcal{L}}{\partial \l_1^2} + 2\left(1-\frac{A}{N} - \l_1 \right) \dfrac{\partial\mathcal{L}}{\partial \l_1} -\frac{A}{N}\left(1-\frac{A}{N}\right)\mathcal{L} = \dfrac{1 - \zeta^A}{\l_1-\l_2} \left( \dfrac{(1-\l_2)^{\frac{A}{N}}}{(1-\l_1)^{\frac{A}{N}}}-1 \right), \\
& \l_2(1-\l_2) \dfrac{\partial^2\mathcal{L}}{\partial \l_2^2} + 2\left(1-\frac{A}{N} - \l_2 \right) \dfrac{\partial\mathcal{L}}{\partial \l_2} -\frac{A}{N}\left(1-\frac{A}{N}\right)\mathcal{L} = \dfrac{1 - \zeta^A}{\l_1-\l_2} \left(1- \dfrac{(1-\l_1)^{\frac{N-A}{N}}}{(1-\l_2)^{\frac{N-A}{N}}} \right). \\
\end{aligned}
\right.
\end{equation}
This is an example of an inhomogeneous Picard-Fuchs differential equation satisfied by normal functions of higher Chow cycles (\cite{dAMS08}, \cite{CDKL}).
Denote the differential operators of the left-hand side of \eqref{PFdiff} by $\mathscr{D}_{\l_1}$ and $\mathscr{D}_{\l_2}$, respectively. Since the periods of $S \to T$ with respect to $\omega$ is annihilated by $\mathscr{D}_{\l_1}$ and $\mathscr{D}_{\l_2}$ (see Proposition \ref{PF}), the function $\mathcal{L}(t)$ is linearly independent from periods, hence we can prove that $\xi_1^{(0)}-\xi_0^{(0)}$ is indecomposable for very general $t\in T$.

In the latter part of this paper (Section \ref{GroupAction}), we consider an action of a finite group $\widetilde{G^2}$ (of order $72N^2$) on the family $S\to T$.
In particular, the subgroup $\widetilde{\Xi}_t$ in Theorem \ref{mainintro} is generated by the images of $\xi_1^{(0)}-\xi_0^{(0)}$ under the $\widetilde{G^2}$-action.
Since $\widetilde{G^2}$ does not act on $T$ trivially, the $\widetilde{G^2}$-action on the normal function is complicated.
Thus, the most part of Section \ref{GroupAction} is devoted to get explicit descriptions of $\widetilde{G^2}$-actions on cycles, relative 2-forms, and
Picard-Fuchs differential operators $\mathscr{D}_{\l_1}$ and $\mathscr{D}_{\l_2}$.
Then we complete the proof of Theorem \ref{mainintro}.

Finally, we mention the difference between Theorem \ref{mainintro} and the result in \cite{Sato}, which treats the case $N=2$ and $A=1$.
In this case, $S_t$ is the Kummer surface and the second author proves $\operatorname{rank} \widetilde{\Xi}_t \geq 18$ for very general $t\in T$ (see \cite[Theorem 7.1]{Sato}). 
Thus, Theorem \ref{mainintro} indicates that $\Xi_t$ becomes larger when $N>2$. 
This difference can be explained as follows:
Let $\Delta$ be a subgroup of $\widetilde{G^2}$ coming from the diagonal action on $\P^1\times \P^1$.
By the action of $\Delta$, we have no additional higher Chow cycles in the case $N=2$, however we obtain additional cycles in the case $N>2$ because of the difference of $\mu_N$-actions on $C_1$ and $C_2$.

\subsection{Outline of the paper}
In Section \ref{HigherChow}, we recall some results from higher Chow cycles and the regulator maps.
In Section \ref{SurfaceFamily}, we define the family of the surfaces $S\to T$ and fix some notations.
In Section \ref{ConstChow}, we explain the construction of families of higher Chow cycles $\xi_0^{(i)}$ and $\xi_1^{(i)}$ for $i \in \Z/N\Z$. 
In Section \ref{computeReg}, we compute the image of $\xi_1^{(0)}-\xi_0^{(0)}$ under the transcendental regulator map.
In Section \ref{GroupAction}, we define the group $\widetilde{G^2}$ and its action on $S\to T$.
Then we calculate the $\widetilde{G^2}$-actions on Picard-Fuchs differential operators.
In Section \ref{main}, we calculate the $\widetilde{G^2}$-actions on the images of normal functions of families of cycles $\xi_0^{(i)}$ and $\xi_1^{(i)}$ under the Picard-Fuchs operators explicitly, we prove Theorem \ref{mainintro}.

\subsection{Notations} Throughout this paper, we fix a positive integer $N\ge 2$ and a primitive $N$th root of unity $\zeta$.
For a positive integer $m$, $\mu_m$ denotes the group of $m$th roots of unity.

\section{Preliminaries} \label{HigherChow}
\subsection{Higher Chow cycles of type (2,1)} 
For a smooth variety $X$ over $\C$, let $\mathrm{CH}^p(X,q)$ be the higher Chow group defined by Bloch.
In this paper, we treat the case $(p,q)=(2,1)$.
In this case, the following fact is well-known (see, e.g., \cite[Corollary 5.3]{MS2}).
\begin{prop}
The higher Chow group $\mathrm{CH}^2(X,1)$ is isomorphic to the middle homology group of the following complex.
\begin{equation*}
K_2^{\mathrm{M}} (\C(X)) \xrightarrow{T} \displaystyle\bigoplus_{C\in X^{(1)}}\C(C)^\times \xrightarrow{\mathrm{div}}\displaystyle\bigoplus_{p\in X^{(2)}}\Z\cdot p 
\end{equation*}
Here $X^{(r)}$ denotes the set of closed subvarieties of $X$ of codimension $r$. 
The map $T$ denotes the tame symbol map from the Milnor $K_2$-group of the function field $\C(X)$.
\end{prop}
Therefore, a higher Chow cycle in $\mathrm{CH}^2(X,1)$ is represented by a formal sum
\begin{equation}\label{formalsum}
\sum_j (C_j, f_j)\in \displaystyle\bigoplus_{C\in X^{(1)}}\C(C)^\times
\end{equation}
where $C_j$ are prime divisors on $X$ and $f_j\in \C(C_j)^\times$ are non-zero rational functions 
on them such that $\sum_j {\mathrm{div}}_{C_j}(f_j) = 0$ as codimension 2
cycles on $X$.

A \textit{decomposable cycle} in $\mathrm{CH}^2(X,1)$ is an element of the image of the group homomorphism
\begin{equation}\label{intersectionproduct}
\mathrm{Pic}(X)\otimes_\Z \Gamma(X,\mathscr{O}_X^\times)=\mathrm{CH}^1(X)\otimes_\Z \mathrm{CH}^1(X,1) \longrightarrow \mathrm{CH}^2(X,1)
\end{equation}
induced by the intersection product. 
Let $C$ be a prime divisor on $X$ and $[C]\in \mathrm{Pic}(X)$ be the class corresponding to $C$. 
The image of $[C]\otimes \alpha$ under (\ref{intersectionproduct}) is represented by $(C,\alpha|_C)$ in the presentation (\ref{formalsum}). 
The cokernel of (\ref{intersectionproduct}) is denoted by $\mathrm{CH}^2(X,1)_{\mathrm{ind}}$. 
For $\xi \in \mathrm{CH}^2(X,1)$, $\xi_{\mathrm{ind}}$ denotes its image in $\mathrm{CH}^2(X,1)_{\mathrm{ind}}$. 
A cycle $\xi$ is called \textit{indecomposable} if $\xi_{\mathrm{ind}}\neq 0$.

\subsection{The regulator map} \label{regulator_map}
By the canonical identification of $\mathrm{CH}^2(X,1)$ with the motivic cohomology $H^3_{\mathcal M}(X,\Z(2))$, there exists the map
\begin{equation}\label{regulatormap}
\mathrm{reg}\colon \mathrm{CH}^2(X,1) \to H^3_{\mathcal D}(X, \Q(2)) = \frac{H^2(X,\C)}{F^2H^2(X,\C)+H^2(X,\Q(2))}
\end{equation}
called the \textit{regulator map}.
The target $H^3_{\mathcal D}(X, \Q(2))$ denotes the Deligne cohomology of $X$. 
We recall an explicit formula for (\ref{regulatormap}) following \cite[pp.458--459]{Levine}. 

Let $X$ be a smooth projective surface over $\C$ such that $H_1(X,\Q)=0$.
By the Poincar\'e duality, the Deligne cohomology of $X$ is isomorphic to the generalized complex torus
\begin{equation}\label{Delignefunctional}
H^3_{\mathcal{D}}(X, \Q(2)) \simeq \frac{(F^1H^2(X,\C))^\vee}{H_2(X,\Q)}
\end{equation}
where $(F^1H^2(X,\C))^\vee$ is the dual $\C$-vector space of $F^1H^2(X,\C)$ and we regard $H_2(X,\Q)$ as a subgroup of $(F^1H^2(X,\C))^\vee$ by the integration. 
Under the isomorphism (\ref{Delignefunctional}), the image of the cycle $\xi$ represented by a formal sum $\sum_j (C_j, f_j)$ under the regulator map is described as follows.

Let $D_j$ be the normalization of the closed curve $C_j$ on $X$. 
Let $\mu_j\colon D_j\to X$ denote the composition of $D_j\to C_j$ and $C_j\to X$. 
We will define a topological 1-chain $\gamma_j$ on $D_j$.
If $f_j$ is a constant, we define $\gamma_j = 0$.
If $f_j$ is not a constant, we regard $f_j$ as a finite morphism from $D_j$ to $\P^1$. 
Then we define $\gamma_j = f_j^{-1}([\infty, 0])$ where $[\infty, 0]$ is a path on $\P^1$ from $\infty$ to $0$ along the positive real axis. 
By the condition $\sum_j \mathrm{div}_{C_j}(f_j) = 0$, $\gamma = \sum_j (\mu_j)_*\gamma_j$ satisfies $\partial \gamma=0$.
In this paper, $\gamma$ is called the \textit{1-cycle associated with $\xi$}.
Since $H_1(X, \Q) = 0$, there exists a 2-chain $\Gamma$ and $m\in \Z_{>0}$ such that $\partial\Gamma = m\cdot\gamma$.
Then the image of $\xi$ under the regulator map is represented by the pairing
\begin{equation}\label{regulatorformula}
\langle \mathrm{reg}(\xi), [\omega]\rangle = \frac{1}{m}\displaystyle\int_\Gamma\omega  + \sum_j\dfrac{1}{2\pi\sqrt{-1}}\displaystyle\int_{D_j-\gamma_j}\log (f_j)\mu_j^*\omega\quad ([\omega]\in F^1H^2(X,\C)).
\end{equation}
Here $\log(f_j)$ is the pull-back of the logarithmic function on $\P^1-[\infty,0]$ by $f_j$.

In this paper, we use the following variant of the regulator map.
\begin{dfn} \label{transreg} 
The \textit{transcendental regulator map} is the composite of the regulator map $(\ref{regulatormap})$ and the projection induced by $H^{2,0}(X) \hookrightarrow F^1H^2(X,\C)$. 
\begin{equation*}
r\colon \mathrm{CH}^2(X,1) \to \frac{F^1H^2(X,\C)^\vee}{H_2(X,\Q)} \twoheadrightarrow \frac{H^{2,0}(X)^\vee}{H_2(X,\Q)}. 
\end{equation*}
We denote this map by $r$. 
\end{dfn}

For a non-zero holomorphic 2-form $\omega$ on $X$, we define the $\Q$-linear subspace $\mathcal{P}(\omega)$ of $\C$ by
\begin{equation*}
\mathcal{P}(\omega) = \left\langle \displaystyle \int_{\Gamma}\omega  \:\: \middle| \:\:\Gamma \text{ is a topological 2-cycles on }X\right\rangle_\Q, 
\end{equation*}
i.e. $\mathcal{P}(\omega)$ is the set of periods of $X$ with respect to $\omega$. 
We denote $\C/\mathcal{P}(\omega)$ by $\mathcal{Q}(\omega)$.
By considering the pairing with $\omega$,
we have a surjective map $H^{2,0}(X)^\vee/H_2(X,\Q) \twoheadrightarrow \C/\mathcal{P}(\omega)=\mathcal{Q}(\omega)$.

By the formula \eqref{regulatorformula}, the image of $\xi \in \mathrm{CH}^2(X,1)$ under the transcendental regulator map is 
\begin{equation}\label{transregval}
\langle r(\xi), [\omega]\rangle  \equiv \frac{1}{m}\cdot\int_{\Gamma}\omega \mod \mathcal{P}(\omega).
\end{equation}
Here, $\Gamma$ is a 2-chain satisfying $\partial \Gamma = m\cdot \gamma$ where $\gamma$ is a 1-cycle associated with $\xi$.
If $\xi$ is decomposable, $r(\xi) = 0$. 
This implies the following.

\begin{prop}\label{transregprop}
If $r(\xi)\neq 0$, we have $\xi_{\mathrm{ind}}\neq 0$. In other words, the transcendental regulator map factors through $\mathrm{CH}^2(X,1)_{\mathrm{ind}}$.
\end{prop}

\subsection{A relative setting}
Since it is difficult to prove non-vanishingness of an element of $H^{2,0}(X)^\vee/H_2(X,\Q)$, we use its relative version.
Let $\pi\colon{\mathcal X}\rightarrow S$ be a smooth family of projective surfaces over a complex manifold $S$. 
We define sheaves $\mathcal{P}$ and $\mathcal{Q}$ of $\Q$-linear spaces on $S$ by 
\begin{equation*}
\begin{aligned}
\mathcal{P} &= \mathrm{Im}(R^2\pi_*\underline{\Q}_ {\mathcal{X}}\to \mathcal{H}om_{\mathscr{O}^{\mathrm{an}}_{S}}(\pi_*\Omega^2_{\mathcal{X}/S}, \mathscr{O}^{\mathrm{an}}_{S})), \\
\mathcal{Q} &= \mathrm{Coker}(R^2\pi_*\underline{\Q}_ {\mathcal{X}}\to \mathcal{H}om_{\mathscr{O}^{\mathrm{an}}_{S}}(\pi_*\Omega^2_{\mathcal{X}/S}, \mathscr{O}^{\mathrm{an}}_{S})),
\end{aligned}
\end{equation*}
where $\underline{\Q}_{\mathcal{X}}$ is the constant sheaf of value $\Q$ on $\mathcal{X}$ and $\mathscr{O}^{\mathrm{an}}_{S}$ is the sheaf of holomorphic functions on $S$. 
Note that $\mathcal{P}$ is a local system on $S$.

For each non-zero relative 2-form $\omega \in \Gamma(\mathcal{X},\Omega^2_{\mathcal{X}/S})$,
we denote the image of $\mathcal{P}$ under the 
projection map
\begin{equation*}
\mathcal{H}om_{\mathscr{O}^{\mathrm{an}}_{S}}(\pi_*\Omega^2_{\mathcal{X}/S}, \mathscr{O}^{\mathrm{an}}_{S}) \twoheadrightarrow  \mathscr{O}^{\mathrm{an}}_{S}
\end{equation*}
by $\mathcal{P}(\omega)$.
The subsheaf $\mathcal{P}(\omega)$ is generated by period functions with respect to $\omega$. 
We also denote the quotient sheaf $\mathscr{O}^{\mathrm{an}}_{S}/\mathcal{P}(\omega)$ by $\mathcal{Q}(\omega)$.
For a local section $\varphi$ of $\mathscr{O}^{\mathrm{an}}_{S}$, its image in $\mathcal{Q}(\omega)$ is denoted by $[\varphi]$. 

For each $s\in S$, the evaluation map $\mathrm{ev}_s\colon \Gamma(S,\mathcal{Q}) \to H^{2,0}(\mathcal{X}_s)^\vee/H_2(\mathcal{X}_s,\Q)$ fits into the following commutative diagram:
\begin{equation*}
\begin{tikzcd}
\Gamma(S,\mathcal{Q}) \arrow[d] \arrow[r,"\mathrm{ev}_s"]& H^{2,0}(\mathcal{X}_s)^\vee/H_2(\mathcal{X}_s,\Q) \arrow[d,twoheadrightarrow] \\
\Gamma(S,\mathcal{Q}(\omega)) \arrow[r] & \C/\mathcal{P}(\omega_s) \arrow[r,equal]&[-40pt]  \mathcal{Q}(\omega_s). \\[-10pt]
{{[f]}} \arrow[r,mapsto] \aru& (f(s) \mod{\mathcal{P}(\omega_s)}) \aru
\end{tikzcd}
\end{equation*}
The following elementary lemma is crucial for the result.
\begin{lem}\label{basiclem}
For a non-zero element $\nu \in \Gamma(S,\mathcal{Q})$, we have $\mathrm{ev}_s(\nu)\neq 0$ for very general $s\in S$.
\end{lem}
\begin{proof}
Since the question is local, we may shrink $S$ in the sense of the classical topology. 
Let $\nu\in \Gamma(S,\mathcal{Q})$ be a non-zero element.
Since $\nu$ is non-zero, there exists a relative 2-form $\omega \in \Gamma(\mathcal{X},\Omega_{\mathcal{X}/S}^2)$ such that $\langle \nu, \omega\rangle \in \mathscr{O}_S^{\mathrm{an}}/\mathcal{P}(\omega) = \mathcal{Q}(\omega)$ is non-zero.
We may assume that there exist a holomorphic function $\varphi\in \Gamma(S,\mathscr{O}_S^{\mathrm{an}})$ such that $\langle \nu, \omega\rangle = [\varphi]$ and a free basis $f_1,f_2,\dots,f_r$ of $\Gamma(S,\mathcal{P}(\omega))$.

For each $\underline{c} = (c_i) \in \Q^r$, we define the holomorphic function $F_{\underline{c}}$ by 
\begin{equation*}
F_{\underline{c}} = \varphi - \sum_{i=1}^r c_if_i.
\end{equation*}
Since $[\varphi]$ is non-zero in $\mathcal{Q}(\omega)$, $F_{\underline{c}}$ is a non-zero holomorphic function for each $\underline{c} \in \Q^r$.
Consider the countable family $\{F_{\underline{c}}\}_{\underline{c}\in \Q^r}$ of the holomorphic functions.
Then outside of the zeros of the functions in this family, $\varphi(s)\not\in \langle f_1(s),f_2(s),\dots, f_r(s) \rangle_\Q = \mathcal{P}(\omega_s)$.
Hence $\mathrm{ev}_s(\nu)\neq 0$.
\end{proof}

Finally, we consider the regulator map in the relative setting.
Suppose that we have irreducible divisors $\mathcal{C}_j$ on $\mathcal{X}$ which are smooth over $S$ and non-zero rational functions $f_j$ on $\mathcal{C}_j$ whose zeros and poles are also smooth over $S$.
Assume that they satisfy the condition $\sum_j\mathrm{div}_{(\mathcal{C}_j)_s}((f_j)_s) = 0$ for each $s\in S$.
Then we have a family of higher Chow cycles $\xi = \{\xi_s\}_{s\in S}$ of type $(2,1)$ such that $\xi_s\in \mathrm{CH}^2(\mathcal{X}_s,1)$ is represented by the formal sum $\sum_{j}((\mathcal{C}_j)_s,(f_j)_s)$.
A family of higher Chow cycles of type $(2,1)$ constructed in this way is called an \textit{algebraic family of higher Chow cycles} in this paper.

If we shrink $S$ in the sense of the classical topology, there exists a $C^\infty$-family of topological 2-chains $\{\Gamma_s\}_{s\in S}$ such that $\Gamma_s$ is a 2-chain associated with $\xi_s$.
If we fix a relative 2-form $\omega$, the function 
\begin{equation*}
S\ni s \mapsto \int_{\Gamma_s}\omega_s \in \C
\end{equation*}
is holomorphic (cf. \cite[Proposition 4.1]{CL}).
Hence we can define the element $\nu_{\mathrm{tr}}(\xi) \in \Gamma(S,\mathcal{Q})$ such that 
\begin{equation*}
\mathrm{ev}_s(\nu_{\mathrm{tr}}(\xi)) = r(\xi_s)
\end{equation*}
for every $s\in S$.
This $\nu_{\mathrm{tr}}(\xi)$ can be regarded as a part of the normal function associated with $\xi$.

\section{The family of surfaces} \label{SurfaceFamily}
Let $R_0=\P^1 \setminus \{0, 1, \infty\}$ and $A_0=\C\left[\l, \frac1{\l(1-\l)} \right]$ be the coordinate ring of $R_0$. 
Let $f_i \colon C_i \to R_0$ $(i=1, 2)$ be a family of smooth projective curves over $R_0$, whose general fiber $(C_i)_{\l}=f_i^{-1}(\l)$ is birational to the singular curve $$y_i^N=x_i^{a_i} (1-x_i)^{b_i} (1- \l x_i)^{c_i}, $$ where $a_i, b_i, c_i \in \{1, \ldots, N-1\}$ and $\gcd(N, a_i, b_i, c_i)=1$. 
We have the natural morphism $C_i \to \P^1 \times R_0$ defined by $(x, y) \mapsto x$.  

Let $T_0$ be the (Zariski) open set of $R_0 \times R_0$ defined by 
$$T_0=\left\{ (\l_1, \l_2) \in R_0 \times R_0 \:\: \middle| \:\:  \l_1\neq \l_2, 1-\l_2, \frac1{\l_2}, \frac1{1-\l_2}, \frac{\l_2-1}{\l_2}, \frac{\l_2}{\l_2-1}\right\}$$
and $B_0$ be the coordinate ring of $T_0$. 
We have the family of surfaces $C_1\times C_2\rightarrow R_0\times R_0$.
We use the same symbol $C_1\times C_2$ for its restriction to $T_0$.
Let $\mathscr{L}$ be the line bundle on $\P^1 \times \P^1 \times T_0$ defined by
\begin{equation*}
\mathscr{L} = pr_1^*\Omega_{\P^1\times R_0/R_0}^1\otimes pr_2^*\Omega_{\P^1\times R_0/R_0}^1.
\end{equation*}
Let $h$ be the global section of $\mathscr{L}^{\otimes(-mN)}$ defined by 
\begin{equation*}
h = x^{a_1}y^{a_2}(1-x)^{b_1}(1-y)^{b_2}(1-\l_1 x)^{c_1} (1-\l_2 y)^{c_2}(dx\otimes dy)^{-mN}
\end{equation*}
where $m$ is an integer greater than\footnote{This condition is necessary for the ramification indices along $x=\infty$ and $y=\infty$ to be positive.} $\max\{\frac{a_1+b_1+c_1}{N},\frac{a_2+b_2+c_2}{N}\}$ and $x,y$ are inhomogeneous coordinates on $\P^1 \times \P^1 \times T_0$.
We define $\overline{S}$ as the cyclic covering of degree $N$ associated with $(\mathscr{L},h)$, i.e., $\overline{S}$ is the relative spectrum of an $\mathscr{O}_{\P^1\times \P^1\times T_0}$-algebra $\bigoplus_{i=0}^{N-1}\mathscr{L}^{\otimes i}$, where the ring structure is defined by 
\begin{equation*}
u\cdot v = \begin{cases}
u\otimes v & (i+j<N)\\
u\otimes v\otimes h & (i+j\geq N)
\end{cases}
\end{equation*}
where $u$ and $v$ are local sections of $\mathscr{L}^{\otimes i}$ and $\mathscr{L}^{\otimes j}$, respectively.

An affine model of $\overline{S}$ is given by
\begin{equation}\label{Sbaraffine}
z^N=x^{a_1}y^{a_2}(1-x)^{b_1}(1-y)^{b_2}(1-\l_1 x)^{c_1} (1-\l_2 y)^{c_2}
\end{equation}
and let $S \to \overline{S}$ be the minimal resolution of singularities.
Let $\mu_N$ be the group of $N$th roots of unity. 
Then $\mu_N$ acts on $C_1 \times C_2$ by 
$$((x_1, y_1), (x_2, y_2)) \mapsto ((x_1, \zeta^{-1} y_1), (x_2, \zeta y_2)), \quad \zeta \in \mu_N. $$
We define the morphism $C_1\times C_2 \to \overline{S}$ by
\begin{equation}\label{CtoS}
((x_1, y_1), (x_2, y_2)) \mapsto (x,y,z) = (x_1,x_2,y_1y_2).
\end{equation}
Then this morphism induces the birational morphism from the quotient $C_1\times C_2/\mu_N$ to $\overline{S}$.
Let $\widetilde{C_1 \times C_2}/\mu_N \to C_1 \times C_2/\mu_N$ be a resolution of singularities and $\widetilde{C_1\times C_2}$ be the fiber product of $C_1\times C_2$ and $\widetilde{C_1 \times C_2}/\mu_N$ over $C_1\times C_2/\mu_N$.
Since $\widetilde{C_1 \times C_2}/\mu_N$ is birational to $\overline{S}$ and $S\to \overline{S}$ is minimal, we have the birational morphism $\widetilde{C_1 \times C_2}/\mu_N\to S$, and this morphism is a composition of blowing-ups.
Then, we have the following commutative diagram: 
\begin{equation}\label{diagram}
  \xymatrix{
     & \ar[ld]_{\eqref{CtoS}} C_1 \times C_2 \ar[d]^{N:1\ \text{cover}} &  \ar[l]_{\text{blow-up}}  \widetilde{C_1 \times C_2} \ar[d]^{N:1\ \text{cover}}   \\
    \overline{S} \ar[d]& \ar[l] C_1 \times C_2/\mu_N & \ar[l]^{\text{blow-up}}  \widetilde{C_1 \times C_2}/\mu_N \ar[d]^{\text{blow-up}} \\
    \P^1 \times \P^1 \times T_0 & &  \ar[llu]^{\text{blow-up}}  \ar[ll] S. 
  }
\end{equation}

In the rest of the paper, we suppose that $$
a_1=b_1=c_1=A, \quad a_2=b_2=c_2=N-A, 
$$ where $A$ is a positive integer satisfying the following conditions.
\begin{align}
    &\gcd(N, A)=1, \label{gcdassumption}\\
    &(N+1)/3 \leq A  \leq (2N-1)/3. \label{rangeassumption}
\end{align}

We define local charts on $\P^1 \times \P^1 \times T_0$ and $S$. 
Let $U=\spec B_0[x, y] \subset \P^1 \times \P^1 \times T_0$ be the affine open subset which is the complement of the divisors $x=\infty$ and $y=\infty$. 
The inverse image of $U$ by $S \to \P^1 \times \P^1 \times T_0$ is covered by two open affine subschemes: 
\begin{align*}
V&=\spec B_0[x, y, v]/(v^Nf(x) -g(y)),  \\
W&=\spec B_0[x, y, w]/(f(x) - w^Ng(y)),  
\end{align*}
where $f(x)=x(1-x)(1-\l_1 x)$ and $g(y)=y(1-y)(1-\l_2 y)$. 
These two subsets are glued by the relation $v=1/w$. 
Then the map $V \cup W$ to the affine model \eqref{Sbaraffine} of $\overline{S}$ is given by 
\begin{align*}
& V \to \overline{S}; \quad (x, y, v) \mapsto (x, y, z)=(x, y, v^{N-A}x(1-x)(1-\l_1 x)),  \\ 
& W \to \overline{S}; \quad (x, y, w) \mapsto (x, y, z)=(x, y, w^Ay(1-y)(1 -\l_2 y)).  
\end{align*}
By these equations, we have only one exceptional curve on $S$ above $(x,y)=(0,0)$ (resp. $(1,1),(0,1)$) and denote it by $Q_{(0,0)}$ (resp. $Q_{(1,1)},Q_{(0,1)}$).

\section{Construction of initial higher Chow cycles} \label{ConstChow}
 
\subsection{Construction of families of higher Chow cycles}
Let 
$$B= B_0[\sqrt[N]{\l_1}, \sqrt[N]{\l_2}, \sqrt[N]{1-\l_1}, \sqrt[N]{1-\l_2}]$$ and $T \to T_0$ be the finite \'etale cover corresponding to $B_0 \to B$. 
We use the same notations of the base changes of the schemes in the diagram \eqref{diagram} by $T \to T_0$. 
Let $D \subset \P^1 \times \P^1 \times T$ be the closed subschemes defined by the local equation $x=y$, and $Z$ be its pull-back by $S \to \P^1 \times \P^1 \times T$. 
On the local chart $V$, $Z \hookrightarrow S$ is described by the following ring homomorphism 
$$B[x, y, v]/(v^Nf(x)-g(y)) \to B [z, v]/(v^N(1-\l_1 z)-(1-\l_2z)); \quad (x, y, v) \mapsto (z, z, v). $$
The curves $Z$ and $Q_{(0, 0)}$ intersect at $N$-points: 
$$Q_{(0, 0)} \cap Z=\{(x, y, v)=(0, 0, \zeta^i) \mid i=0, \ldots, N-1 \}. $$
The curves $Z$ and $Q_{(1, 1)}$ intersect at $N$-points: 
$$Q_{(1, 1)} \cap Z= \{(x, y, v)=(1, 1,  \zeta^i \sqrt[N]{1-\l_2}/\sqrt[N]{1-\l_1}) \mid i =0, \ldots, N-1\}. $$ 

We define rational functions $\psi^{i}_{\bullet} \in \C(Z)$ and $\varphi^{i}_{\bullet} \in \C(Q_{(\bullet, \bullet)})$ $(\bullet \in \{0, 1\}, i \in \Z/N\Z)$ by the following equations: 
\begin{align*}
& \psi_0^{(i)}=(v-\zeta^{i+1})\cdot (v-\zeta^{i})^{-1},   \\
& \varphi^{(i)}_0=(v-\zeta^i)\cdot (v-\zeta^{i+1})^{-1},  \\
& \psi^{(i)}_1=\left(v-\zeta^{i+1} \dfrac{\sqrt[N]{1-\l_2}}{\sqrt[N]{1-\l_1}} \right) \cdot \left(v- \zeta^i \dfrac{\sqrt[N]{1-\l_2}}{\sqrt[N]{1-\l_1}} \right)^{-1},  \\
&\varphi^{(i)}_1=\left(v- \zeta^{i}\dfrac{\sqrt[N]{1-\l_2}}{\sqrt[N]{1-\l_1}} \right) \cdot \left(v-\zeta^{i+1} \dfrac{\sqrt[N]{1-\l_2}}{\sqrt[N]{1-\l_1}} \right)^{-1}.  
\end{align*}
\begin{dfn} \label{cycledef}
For $i \in \Z/N\Z$, 
define algebraic families of higher Chow cycles $\xi^{(i)}_{\bullet} =\{(\xi^{(i)}_{\bullet})_t\}_{t \in T} \in \ch^2(S, 1)$ by 
\begin{align*}
&(\xi_0^{(i)})_t=(Z_t, \psi_0^{(i)}) +((Q_{(0, 0)})_t, \varphi_0^{(i)}), \\
&(\xi_1^{(i)})_t=(Z_t, \psi_1^{(i)}) +((Q_{(1, 1)})_t, \varphi_1^{(i)}),  
\end{align*}
where $Z_t$ and $(Q_{(\bullet,\bullet)})_t$ denote the fiber of $Z$ and $Q_{(\bullet,\bullet)}$ over $t\in T$.

\end{dfn}

\section{Computation of the regulator} \label{computeReg}
Put $\mathcal{A}=C_1 \times C_2$. 
Then we have  $T$-morphisms 
$$\mathcal{A} \xleftarrow[\text{blowing-up}]{p} \widetilde{\mathcal{A}} \xrightarrow[N:1 \ \text{quotient}]{\pi} S. $$
We name the morphisms $p$ and $\pi$ as above. 
Let $pr_i \colon C_1 \times C_2 \to C_i$ be the $i$th projection.  
Let $\omega_1$ (resp. $\omega_2$) be the $1$-form on $C_1$ (resp. $C_2$) over $T$ defined by 
\begin{align*}
&\omega_1=\dfrac{dx_1}{y_1} \quad \left(\text{resp. } \omega_2=\dfrac{dx_2}{y_2}\right), 
\end{align*}
which is a holomorphic $1$-form on $C_1$ (resp. $C_2$) over $T$ by the assumption \eqref{rangeassumption} (see \cite[Section 6.1]{{Archinard}}).  
We have the relative $2$-form $pr_1^* (\omega_{1}) \wedge pr_2^*(\omega_{2})$ on $\mathcal{A} \to T$. 
Then its pull-back by $p$ is stable under the covering transformation of $\pi$, hence it descends to $S \to T$.  
We denote the resulting $2$-form on $S$ by $\omega$, which is described as 
$$\omega=\frac{dx \wedge dy}{v^{N-A} f(x)}=\frac{dx \wedge dy}{w^Ag(y)}$$
on the local charts $V$ and $W$, respectively. 

\begin{prop} \label{PF}
Let $\omega$
be the above relative $2$-form on the family $S \to T$. 
Let $\mathscr{ D}_{\l_i} \colon \mathscr{O}_T^{\rm an} \to \mathscr{O}_T^{\rm an}$ $(i=1, 2)$ be the differential operator defined by 
\begin{align*}
&\mathscr{D}_{\l_1}=\l_1(1-\l_1) \dfrac{\partial^2}{\partial \l_1^2} + 2\left(1-\frac{A}{N} - \l_1 \right) \dfrac{\partial}{\partial \l_1} -\frac{A}{N}\left(1-\frac{A}{N}\right), \\
&\mathscr{D}_{\l_2}=\l_2(1-\l_2) \dfrac{\partial^2}{\partial \l_2^2} +2\left(\frac{A}{N} - \l_2 \right) \dfrac{\partial}{\partial \l_2} -\frac{A}{N}\left(1-\frac{A}{N}\right). 
\end{align*}
Let $\mathscr{D}=
\begin{pmatrix} 
\mathscr{D}_{\l_1} \\
\mathscr{D}_{\l_2} 
\end{pmatrix}
\colon \mathscr{O}_T^{\rm an} \to ( \mathscr{O}_T^{\rm an})^{\oplus 2}$. 
Then for any local section $f$ of $\mathcal{P}({\omega}) \subset \mathscr{O}^{\rm an}_T$, we have $\mathscr{D}(f)=0$, which induces the morphism 
$\mathcal{Q}({\omega}) \to (\mathscr{O}^{\rm an}_T)^{\oplus 2}$. 
\end{prop}

\begin{proof}
For $t \in T$, let $\phi$ be the morphism of Hodge structures defined by 
$$\phi \colon H^2(\mathcal{\mathcal{A}}_t) \xrightarrow{p^*} H^2(\widetilde{\mathcal{A}}_t) \xrightarrow{\pi_!} H^2(S_t), $$
where $\pi_!$ is the Gysin morphism induced by $\pi$. 
Since the degree of $\pi$ is $N$, we have $\pi_! \circ \pi^* \colon H^2(S_t) \to H^2(S_t)$ equals multiplication $N$ (cf. \cite[Remark 7.29]{Voi}).   

Let $f \in \mathcal{P}(\omega)$ be a local section. Then there exists $[\Gamma_t] \in H_2(S_t, \Z)$ such that 
$$f=\int_{\Gamma_t} \omega_t. $$
Since $\pi^*(\omega)=pr_1^*(\omega_1) \wedge pr_1^*(\omega_2)$, we have 
$$\phi([pr_1^*(\omega_1) \wedge pr_1^*(\omega_2)]_t)=N[\omega_t]. $$
Therefore, we obtain 
\begin{align} \label{period}
\int_{\Gamma_t} \omega_t=\frac1N \int_{\phi^{\vee}(\Gamma_t)}(pr_1^*(\omega_1) \wedge pr_2^*(\omega_2))_t, 
\end{align}
where $\phi^{\vee} \colon H_2(S_t) \to H_2(\mathcal{A}_t)$ is the dual of $\phi$. 
The right-hand side of \eqref{period} vanishes under the action $\mathscr{D}_{\l_1}$ and $\mathscr{D}_{\l_2}$ since $\mathscr{D}_{\l_1} \omega_1$ and $\mathscr{D}_{\l_2} \omega_2$ are exact forms, hence the left-hand side also vanishes under the action $\mathscr{D}$.  
Hence the assertion follows. 
\end{proof}

\subsection{Construction of topological chains}
In this subsection, we construct a topological 2-chain for computing the images of $\xi_{1}^{(0)}-\xi_0^{(0)}$ under the transcendental regulator map.
To use the results in Section \ref{HigherChow}, we should prove the following.
\begin{prop}\label{h1zero}
For each $t\in T$, the surface $S_t$ satisfies $H^1(S_t,\Q)=0$.
\end{prop}
\begin{proof}
Consider the morphism of singular cohomologies 
\begin{equation*}
H^1(S_t,\Q) \xrightarrow{\pi^*} H^1(\widetilde{\mathcal{A}}_t,\Q) \xrightarrow{\pi_!} H^1(S_t,\Q).
\end{equation*}
This map coincides with the multiplication by $N$.
Thus it is enough to show that $\mathrm{Im}(\pi^*)=0$.
Since $\mathrm{Im}(\pi^*)$ is contained in the $\mu_N$-invariant part $H^1(\widetilde{\mathcal{A}}_t,\Q)^{\mu_N}$, we will show that $H^1(\widetilde{\mathcal{A}}_t,\Q)^{\mu_N}=0$.
The blowing-up does not affect $H^1$, it suffices to show that $H^1(\mathcal{A}_t,\Q)^{\mu_N}=0$.
We have the K\"unneth decomposition 
\begin{equation}\label{Kunneth}
H^1(\mathcal{A}_t,\Q)=H^1(C_1,\Q)\otimes H^0(C_2,\Q)\oplus H^0(C_1,\Q)\otimes H^1(C_2,\Q)
\end{equation}
By the explicit description of the $\mu_N$-action on $H^1(C_i,\Q)$ in \cite{Archinard}, the invariant part of right-hand side of \eqref{Kunneth} is trivial.
Thus we have $H^1(\mathcal{A}_t,\Q)^{\mu_N}=0$.
\end{proof}

For a while, we fix $t\in T$ so that $\lambda_i$ satisfies $0<\lambda_i<1$ and $\sqrt[N]{\lambda_i}, \sqrt[N]{1-\lambda_i}$ coincides with the ordinary $N$th root for positive numbers for $i=1,2$.

Let $\triangle$ be the locally closed subset of $\R^2$ defined by
\begin{equation*}
\triangle = \{(x,y)\in \R^2\mid 0< y\le x < 1\}.
\end{equation*}
We regard $\triangle$ as a subset of $\P^1\times \P^1$.
Since $\triangle$ does not intersect with the branching locus of $S_t\to \P^1\times \P^1$, the inverse image of $\triangle$ by $S_t\to \P^1\times \P^1$ decomposes into $N$ pieces of locally closed subsets $\triangle_0,\triangle_1,\dots, \triangle_{N-1}$, each of which is homeomorphic to $\triangle$.
We label them so that the closure of $\triangle_i$ contains $(x,y,v)=(0,0,\zeta^i) \in V$.

More explicitly, $\triangle_i$ is the  image of the following map.
\begin{align*}
\triangle_i \to V (\subset S_t); \quad (x, y) \mapsto (x, y, v)= \left(x, y, \zeta^i\dfrac{\sqrt[N]{y(1-y)(1-\l_2 y)}}{\sqrt[N]{x(1-x)(1- \l_1 x)}}\right),
\end{align*}
where $\sqrt[N]{a}$ denotes the ordinary $N$th root for a positive real number $a$.
Let $K_i$ be the closure of $\triangle_i$, which is a compact subset on $S_t$.
We can regard $K_i$ as a 2-chain on $S_t$.

\begin{prop}
For a holomorphic 2-form $\eta \in \Gamma(S_t,\Omega_{S_t}^2) = H^{2,0}(S_t)$, we have 
\begin{equation*}
\langle r((\xi_1^{(0)})_t - (\xi_0^{(0)})_t), [\eta] \rangle \equiv \int_{K_0} \eta -\int_{K_1} \eta\pmod{\mathcal{P}(\eta)}. 
\end{equation*}
\end{prop}
\begin{proof}
First, we will describe the boundaries of $K_i$.
We label the paths appearing in the boundary of $K_i$ as in Figure \ref{Kfigure}.

\begin{figure}[h]
\centering
\includegraphics[width=10cm]{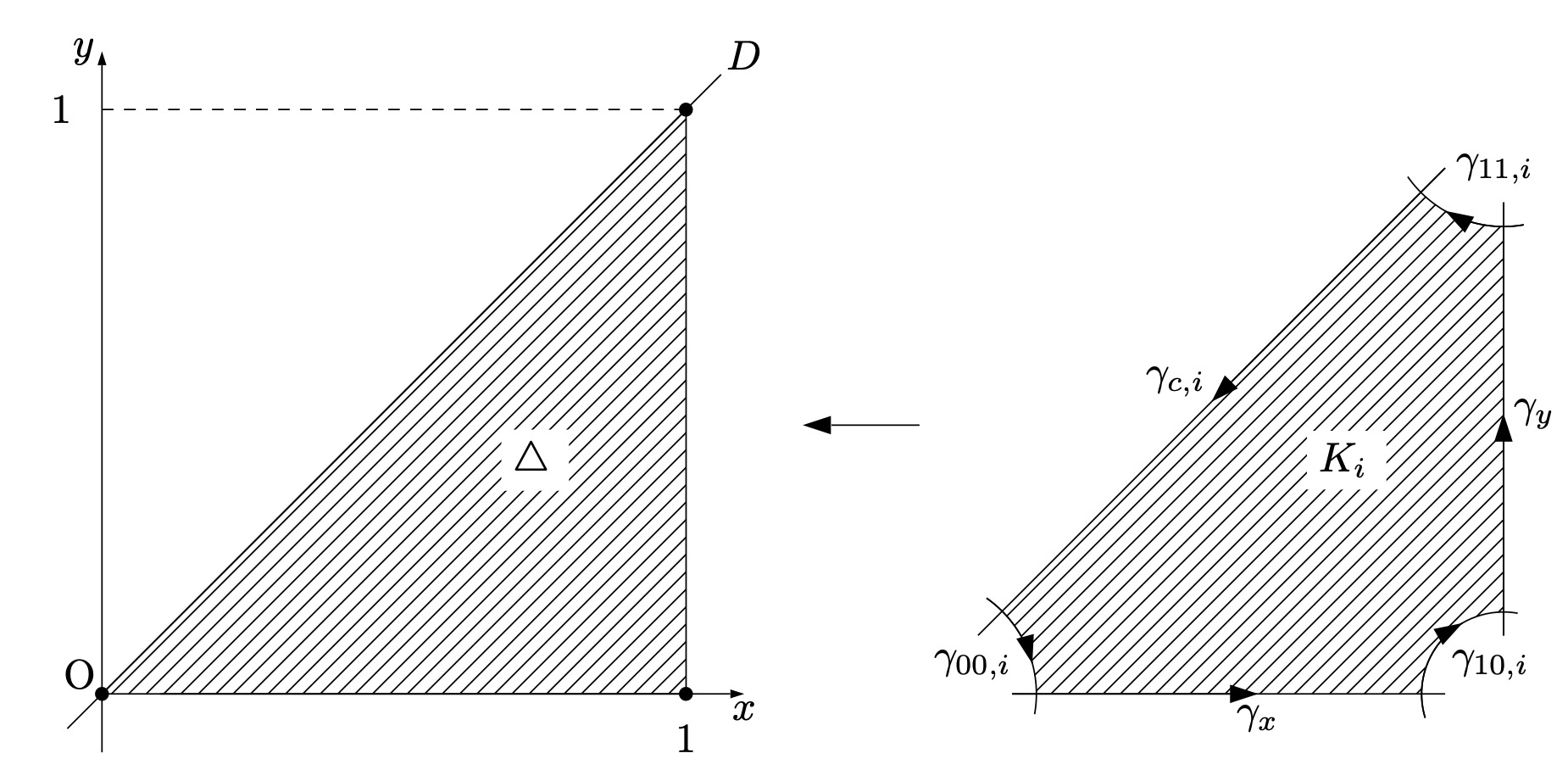}
\caption{The topological 2-chain $K_i$}
\label{Kfigure}
\end{figure}

The following properties hold.
\begin{enumerate}
\renewcommand{\labelenumi}{(\alph{enumi})}
\item The paths $\gamma_x$ and $\gamma_y$ are contained in the branching locus of $S_t\to \P^1\times \P^1$, so they are in common for every $i$.
\item The path $\gamma_{c,i}$ is a path from $(x,y,v) = \left(1,1,\zeta^i\dfrac{\sqrt[N]{1-\lambda_2}}{\sqrt[N]{1-\lambda_1}}\right)$ to $(x,y,v) = (0,0,\zeta^i)$ on the curve $Z_t$.
\item The path $\gamma_{00,i}$ (resp.=$\gamma_{11,i}, \gamma_{10,i}$) is a path on the exceptional curve $(Q_{(0,0)})_t$ (resp.~$(Q_{(1,1)})_t, (Q_{(1,0)})_t$).
\end{enumerate}
Since the starting point and the ending point of $\gamma_{10,i}$ are the same for $i=0,1$, we have $\partial(\gamma_{10,1}-\gamma_{10,0})=0$.
The exceptional curve $(Q_{(1,0)})_t$ is isomorphic to $\P^1$ and $H_1(\P^1,\Z)=0$, so we have a 2-chain $\Gamma_{1,0}$ on $(Q_{(1,0)})_t$ such that $\partial \Gamma_{1,0} = \gamma_{10,1}-\gamma_{10,0}$.

Let $\gamma_c$ (resp.~$\gamma_{00},\gamma_{11}$) be the 1-chain on $Z_t$ (resp. $(Q_{(0,0)})_t, (Q_{(1,1)})_t$) which is the pull-back of $[\infty, 0]$ on $\P^1$ (see Section \ref{regulator_map}) by the rational function $\psi_1^{(0)}\cdot (\psi_0^{(0)})^{-1} \in \C(Z_t)^\times$ (resp. $\varphi_0^{(0)}\in \C((Q_{(0,0)})_t)^\times, (\varphi_1^{(0)})^{-1}\in \C((Q_{(1,1)})_t)^\times$).
Then $\gamma_{c}+\gamma_{00}+\gamma_{11}$ is the 1-cycles associated with $\xi_{1}^{(0)} - \xi_{0}^{(0)}$.
Furthermore, we have
\begin{equation*}
\begin{aligned}
&\partial ( \gamma_{c,1} + \gamma_c -\gamma_{c,0}) = 0, \\
&\partial (\gamma_{00,1}+ \gamma_{00}  -\gamma_{00,0}) = 0,\\
&\partial ( \gamma_{11,1}+\gamma_{11} -  \gamma_{11,0}) = 0.\\
\end{aligned}
\end{equation*}
Since $Z_t,(Q_{(0,0)})_t$, and $(Q_{(1,1)})_t$ are isomorphic to $\P^1$, we have 2-chains $\Gamma_{c},\Gamma_{0,0}$, and $\Gamma_{1,1}$ on curves $Z_t,(Q_{(0,0)})_t$, and $(Q_{(1,1)})_t$, respectively,  such that
\begin{equation*}
\begin{aligned}
&\partial\Gamma_c  =  \gamma_{c,1} + \gamma_c -\gamma_{c,0},\\
&\partial\Gamma_{0,0} =\gamma_{00,1}+ \gamma_{00}  -\gamma_{00,0},\\
&\partial\Gamma_{1,1} = \gamma_{11,1}+\gamma_{11} -  \gamma_{11,0}.\\
\end{aligned}
\end{equation*}
By above relations, we have
\begin{equation*} 
\partial (K_0-K_1 + \Gamma_c + \Gamma_{0,0}+\Gamma_{1,0}+\Gamma_{1,1}) =  \gamma_{c}+\gamma_{00}+\gamma_{11}.
\end{equation*}
Hence the 2-chain $K_0-K_1 + \Gamma_c + \Gamma_{0,0}+\Gamma_{1,0}+\Gamma_{1,1}$ is a 2-chain associated with $\xi_{1}^{(0)}-\xi_{0}^{(0)}$. 
By the formula \eqref{transregval}, we have 
\begin{equation}\label{integral}
\langle r((\xi_1^{(0)})_t - (\xi_0^{(0)})), [\eta]\rangle = \int_{K_0}\eta  - \int_{K_1}\eta + \int_{\Gamma_c}\eta + \int_{\Gamma_{0,0}}\eta + \int_{\Gamma_{1,0}}\eta +\int_{\Gamma_{1,1}}\eta
\end{equation}
for each $\eta\in H^{2,0}(S_t)$.
Since $\eta$ is holomorphic 2-form, its restrictions on $(Q_{(0,0)})_t$, $(Q_{(1,1)})_t$, $(Q_{(1,0)})_t$ and $Z_t$ vanish.
Hence the terms after the third term in the equation (\ref{integral}) are zero.
Thus the proposition follows.
\end{proof}

\subsection{A differential equation for the normal function}

Using the local description of $\omega_t$, we have 
$$\int_{K_i}\omega_t= \int_{\triangle_i} \dfrac{dx dy}{\sqrt[N]{x^A(1-x)^A(1-\l_1 x)^A} \sqrt[N]{y^{N-A} (1-y)^{N-A} (1- \l_2 y)^{N-A}}}. $$

Define $\mathcal{L}(t)$ using the integral expression
$$\mathcal{L}(t)= \int_{K_0} \omega_t -\int_{K_{1}} \omega_t, $$
which is a multi-valued holomorphic function (cf. \cite[Proposition 5.3 (1)]{Sato}). 

\begin{lem}\label{covering}
For $i \in \Z/N\Z$, we have 
\begin{equation*}
\int_{K_i}\omega_t = \zeta^{Ai}\int_{K_0}\omega_t.
\end{equation*}
\end{lem}
\begin{proof}
Consider the automorphism 
$$\sigma_i\colon S\to S;\quad (x,y,v)\mapsto (x,y,\zeta^i v). $$
Then we have $\sigma_i^*\omega =\zeta^{Ai}\omega$ and $\sigma_i(K_0) = K_i$. Thus the lemma follows.
\end{proof}

\begin{thm} \label{main:1}
Let $\mathscr{D}$ be the differential operator defined in Proposition \ref{PF}.
The multi-valued holomorphic function $\mathcal{L}(t)$ satisfies the following system of differential equations: 
$$\mathscr{D}(\mathcal{L})=
\frac{1 - \zeta^A}{\l_1-\l_2} 
\begin{pmatrix}
\dfrac{(1-\l_2)^{\frac{A}{N}}}{(1-\l_1)^{\frac{A}{N}}}-1\\
1- \dfrac{(1-\l_1)^{\frac{N-A}{N}}}{(1-\l_2)^{\frac{N-A}{N}}}
\end{pmatrix}. $$
In particular, $\mathscr{D}(\langle \nu_{\rm tr}(\xi_1^{(0)} -\xi_0^{(0)}), \omega  \rangle)$ coincides with the right-hand side. 
\end{thm}

\begin{proof}
Let $H(\l_1, x)$ be the local holomorphic function defined by 
$$H(\l_1, x)=-\frac{A}{N} x^{1-\frac{A}N}(1-x)^{1-\frac{A}N} (1-\l_1x)^{-1- \frac{A}N}. $$
Then we have 
$$\mathscr{D}_{\l_1}\left( \frac{1}{x^{\frac{A}N}(1-x)^{\frac{A}N}(1-\l_1x)^{\frac{A}N} }\right)=\frac{\partial }{\partial \l_1} H(\l_1, x). $$
Therefore, by Stokes's theorem, we have 
\begin{align*}
\mathscr{D}_{\l_1} \left(\int_{K_0} \omega_t\right) &=\mathscr{D}_{\l_1} \left(\int_{K_0} \dfrac{dx dy}{\sqrt[N]{x^A(1-x)^A(1-\l_1 x)^A} \sqrt[N]{y^{N-A} (1-y)^{N-A} (1- \l_2 y)^{N-A}}} \right) \\
&=\int_{K_0} d \left( \dfrac{H(\l_1, x)}{\sqrt[N]{y^{N-A} (1-y)^{N-A} (1- \l_2 y)^{N-A}}} \right) \\
&=\int_{\partial K_0} \dfrac{H(\l_1, x)}{\sqrt[N]{y^{N-A} (1-y)^{N-A} (1- \l_2 y)^{N-A}}}. 
\end{align*}
Since the $1$-form vanishes on $\partial K_0$ except $\gamma_{c, 0}$, we have 
\begin{align*}
=\frac{A}N \int_0^1 \dfrac{dz}{(1-\l_2 z)^{1-\frac{A}N} (1-\l_1 z)^{1+\frac{A}{N}} } 
&=\frac{A}{\l_1-\l_2}\int_1^{\frac{\sqrt[N]{1-\l_2}}{\sqrt[N]{1-\l_1}}} u^{A-1}du  =\frac{1}{\l_1-\l_2} \left( \frac{(1-\l_2)^{\frac{A}{N}}}{(1-\l_1)^{\frac{A}{N}}}-1\right). 
\end{align*}
Here, we put $u=\frac{\sqrt[N]{1-\l_2z}}{\sqrt[N]{1-\l_1z}}$. 
Then by Lemma \ref{covering}, we can check the coincidence of the first component in the statement. 
The coincidence of the second component is similarly proved.  
\end{proof}

\begin{cor}
For a general $t \in T$,  we have 
$$\rk \ch^2(S_t, 1)_{\rm ind} \geq \varphi(N), $$
where $\varphi$ is Euler's totient function. 
\end{cor}
\begin{proof}
Using the automorphism $\sigma_i$ in Lemma \ref{covering}, we have $(\sigma_i)^*(\xi^{(i)}_j) = \xi^{(0)}_j$ for $j=0,1$ by the explicit description of rational functions.
Since $\sigma_i(K_0) = K_i$, $K_i$ is a 2-chain related to $\xi^{(i)}_1- \xi^{(i)}_0$.
Combining with this fact and Lemma \ref{covering}, we have $$\nu_{\mathrm{tr}}(\xi_1^{(i)}-\xi_0^{(i)})=\zeta^{Ai}\nu_{\mathrm{tr}}(\xi_1^{(0)}-\xi_0^{(0)}).$$
Thus $\langle \xi^{(i)}_1-\xi^{(i)}_0 \mid i\in \Z/N\Z \rangle_\Z$ is a $\Z[\zeta]$-module generated by $\xi_1^{(0)}-\xi_0^{(0)}$.
By Theorem \ref{main:1}, $\xi_1^{(0)}-\xi_0^{(0)}$ is non-torsion.
Then, as a module over the ring $\Z[\zeta]$, this module is free of rank 1, so the rank of this as an abelian group is equal to $\varphi(N)$.
Therefore, the cycles $(\xi^{(i)}_1)_t - (\xi^{(i)}_0)_t$ generate a subgroup of rank at least $\varphi(N)$ in the indecomposable part.
\end{proof} 

\section{A group action on varieties} \label{GroupAction}
In this section, we define a group $\widetilde{G^2}$ and its action on $S\to T$.
Then we describe $\widetilde{G^2}$-actions on period functions and normal functions.

\subsection{Notations}

For an automorphism $\rho$ on $X$, $\rho^\sharp$ denotes the morphism of sheaves of rings 
$$\mathscr{O}_X\to \rho_*\mathscr{O}_X;\quad  \varphi\mapsto \varphi\circ\rho.$$
We use the same notation $\rho^\sharp$ for its adjoint $\rho^*\mathscr{O}_X\to \mathscr{O}_X$.

If a group $G$ acts on $X$, a \textit{$G$-linearization} on an line bundle $\L$ on $X$ is a collection $(\Phi_{g}\colon  g^*\L\xrightarrow{\sim} \L)_{g\in G}$ of $\mathscr{O}_X$-module isomorphisms satisfying the relation $\Phi_{h}\circ h^*(\Phi_g) = \Phi_{gh}$ for $g,h\in G$.

A \textit{1-cocycle} on $\Gamma(X,\mathscr{O}_X^\times)$ is a map $\chi\colon G\rightarrow \Gamma(X,\mathscr{O}_X^\times)$ satisfying
\begin{equation*}
\chi(gh) = h^\sharp(\chi(g))\cdot \chi(h)
\end{equation*}
for any $g,h\in G$. 

\subsection{Construction of a group action on $S\to T$}
Let $G_0$ be the symmetric group of the set 
$\{0,1,1/\lambda\}$.
We define the $G_0$-action on $R_0$ and $\P^1\times R_0$ as in Table \ref{groupactiontable}.

\begin{table}[H]
\begin{tabular}{c|c|c}
$G_0$ & action on $R_0$ & action on $\P^1$ \\ \hline 
$(0\:\:1)$ & $\lambda\mapsto \frac{\lambda}{\lambda-1}$ & $z\mapsto 1-z$\\
$(0\:\:1/\lambda)$ & $\lambda \mapsto 1-\lambda$ & $z\mapsto \frac{1-\lambda z}{1-\lambda}$\\
$(1\:\:1/\lambda)$ & $\lambda \mapsto 1/\lambda$ & $z\mapsto \lambda z$\\
$(0\: \:1\:\: 1/\lambda)$ & $\lambda \mapsto \frac{1}{1-\lambda}$ & $z\mapsto 1-\lambda z$\\
$(0\:\: 1/\lambda \: \: 1)$ & $\lambda \mapsto \frac{\lambda-1}{\lambda}$ & $z\mapsto \frac{\lambda(1-z)}{z}$
\end{tabular}
\caption{A $G_0$-action on $\P^1\times R_0$}\label{groupactiontable}
\end{table}

By definition, the second projection $\P^1\times R_0\rightarrow R_0$ is equivariant with respect to these $G_0$-actions.
Furthermore, $G_0$ naturally acts on the set of the sections $z = 0, 1, 1/\lambda$ of $\P^1\times R_0\rightarrow R_0$, and for each element in $G_0$, its action on the set of these sections coincides with its label in $G_0$.
Let $A=A_0\left[\sqrt[N]{\lambda},\sqrt[N]{1-\lambda}\right]$ and $R=\spec A$.
\begin{prop}
There exists a group $G$, a surjective group homomorphism $p\colon G\to G_0$, and a $G$-action on $R$ satisfying the following property.
\begin{enumerate}
\item The kernel of $p$ is isomorphic to $\mu_N\times \mu_N$.
We identify $\ker(p)$ with $\mu_N\times \mu_N$ through this isomorphism.
\item The subgroup $\ker(p)=\mu_N\times \mu_N$ acts on $A$ as follows.
\begin{equation*}
\ker(p) \times A \rightarrow A; ((\zeta^i, \zeta^j),(\sqrt[N]{\lambda}, \sqrt[N]{1-\lambda}))\mapsto (\zeta^i\sqrt[N]{\lambda}, \zeta^j\sqrt[N]{1-\lambda})
\end{equation*}
\item The natural morphism $R\rightarrow R_0$ is equivariant with respect to these $G$-actions, where $G$ acts on $S_0$ through $p\colon G\to G_0$.
\end{enumerate}
For $\rho \in G$, we denote $p(\rho)\in G_0$ by $\underline{\rho}$.
\end{prop}
\begin{proof}
Since $\Aut(A/A_0) =\mu_N\times \mu_N$, it is enough to show that for each $\rho \in G_0$, there exists a lift $\tilde{\rho}\in \Aut(R)$ of $\rho$.
It is enough to check the above property for generators $(0 \ 1), (0 \ 1/\lambda) \in G_0$.
We can directly construct their lifts, since $\C$ contains a primitive $2N$th root of unity.
\end{proof}

Since $\P^1\times R\rightarrow R$ is the base change of $\P^1\times R_0\rightarrow R_0$, the $G_0$-action on $\P^1\times R_0\to R_0$ and the $G$-action on $R$ induces a $G=G_0\times_{G_0} G$-action on $\P^1\times R\rightarrow R$.
By taking the self product of $\P^1\times R\rightarrow R$, we have a $G^2$-action on $(\P^1\times \P^1)\times (R\times R)\rightarrow R\times R$.
Furthermore, since the open subset $T$ of $R\times R$ is stable under $G^2$-action, we have a $G^2$-action on $\P^1\times \P^1\times T\to T$.

Next, we consider the lift of $G^2$-action on $\overline{S}$.
We use the following liftability criterion of group actions.

\begin{prop}[{\cite[Proposition 6.3]{Sato}}]\label{liftprop}
Suppose a group $G$ acts on a variety $X$ and a $G$-linearization $(\Phi_g)_{g\in G}$ of a line bundle $\L$ on $X$ are given.
Let $\pi \colon Y\rightarrow X$ be the cyclic covering of degree $m$ associated with $(\L,s)$.
Assume that there exists a 1-cocycle $\chi \colon G\rightarrow \Gamma(X,\mathscr{O}^\times_X)$ such that 
\begin{equation*}
\Phi_g^{\otimes (-m)}(g^*(s)) = \chi(g)^{-m}\cdot s
\end{equation*}
for any $g\in G$.
Then there exists a $G$-action on $Y$ such that $\pi$ is $G$-equivariant.
More explicitly, the $G$-action on $Y$ is induced by the $\mathscr{O}_X$-module isomorphism 
\begin{equation*}
\begin{tikzcd}
g^*\L  \arrow[r,"\Phi_g","\sim"'] & \L \arrow[r,"\chi(g)^{-1}\times ","\sim"'] &[25pt] \L.
\end{tikzcd}
\end{equation*}
\end{prop}

To use this proposition, we will describe the $G^2$-linearization on $\L$ and $h$ defined in Section \ref{SurfaceFamily}.
Since $\P^1\times R\rightarrow R$ is $G$-equivariant, we have a natural $G$-linearization on $\Omega^1_{\P^1\times R/R}$.
This $G$-linearization induces a $G^2$-linearization $(\Phi_{\rho})_{\rho \in G^2}$ on $\L$ described as 
\begin{equation*}
\Phi_{\rho}(\rho^*(dx\otimes dy)) = \left(\frac{\del }{\del x}\rho_1^\sharp(x)\right)\left(\frac{\del }{\del y}\rho_2^\sharp(y)\right)\cdot (dx\otimes dy)\quad (\rho = (\rho_1,\rho_2)\in G^2).
\end{equation*}
Using this description, the 1-cocycle $\eta$ defined by 
\begin{equation*}
\Phi_\rho^{\otimes(-N)}(\rho^*(h)) = \eta(\rho)^{-1}\cdot h
\end{equation*}
is decribed as 
\begin{equation*}
\eta(\rho) = \eta_1(\underline{\rho}_1)\cdot \eta_2(\underline{\rho}_2), 
\end{equation*}
where $\eta_1$ (resp. $\eta_2$) is a 1-cocycle determined by the image of $G^2 \xrightarrow{pr_1} G\xrightarrow{p} G_0$ (resp. $G^2 \xrightarrow{pr_2} G\xrightarrow{p} G_0$) and given by Table \ref{1cocycleeta}.

\begin{table}[h]
\begin{tabular}{c|c||c|c}
$\underline{\rho}_1$ & $\eta_1(\underline{\rho}_1)$  & $\underline{\rho}_1$ &$\eta_1(\underline{\rho}_1)$  \\ \hline
$\mathrm{id}$ & $ 1 $ & $(0\:1)$ &  $(-1)^N(1-\lambda_1)^A$ \\ 
$(0\:1/\lambda)$ & $(-1)^{N-A}\lambda_1^{N-2A}(1-\lambda_1)^{2A-N}$ &$(0\:1\:1/\lambda)$ &$(-1)^{N-A}(1-\lambda_1)^{A}\lambda_1^{N-2A}$  \\

$(1\:1/\lambda)$ & $\lambda_1^{N-A}$ & $(0\:1/\lambda\:1)$& $ (-1)^A\lambda_1^{N-A}(1-\lambda_1)^{2A-N}$ \\
\multicolumn{4}{c}{}\\
$\underline{\rho}_2$ & $\eta_2(\underline{\rho}_2)$  & $\underline{\rho}_2$& $\eta_2(\underline{\rho}_2)$ \\ \hline
$\mathrm{id}$ & $ 1 $ & $(0\:1)$ &  $(-1)^N(1-\lambda_2)^{N-A}$ \\
$(0\:1/\lambda)$ & $(-1)^{A}\lambda_2^{2A-N}(1-\lambda_2)^{N-2A}$ &$(0\:1\:1/\lambda)$ &$(-1)^{A}(1-\lambda_2)^{N-A}\lambda_2^{2A-N}$ \\
$(1\:1/\lambda)$ & $\lambda_2^{A}$ & $(0\:1/\lambda \:1)$& $(-1)^{N-A}\lambda_2^{A}(1-\lambda_2)^{N-2A}$\\
\end{tabular}
\vspace{10pt}

\caption{The 1-cocycle $\eta_1,\eta_2$}\label{1cocycleeta}
\end{table}

To use Proposition \ref{liftprop}, we will find a ``$N$th root" of the 1-cocycle $\eta$.
We define coboundary 1-cocycles $\phi_1,\phi_2 \colon G^2\rightarrow B^\times$ by

\begin{equation*}
\begin{aligned}
&\phi_1(\rho) = \rho_1^\sharp\left(\frac{\lambda_1^{\frac{A}{N}}(1-\lambda_1)^{\frac{N-A}{N}}}{\lambda_1^2-\lambda_1+1}\right)\cdot \left(\frac{\lambda_1^2-\lambda_1+1}{\lambda_1^{\frac{A}{N}}(1-\lambda_1)^{\frac{N-A}{N}}}\right),\\
&\phi_2(\rho) = \rho_2^\sharp\left(\frac{\lambda_2^{\frac{N-A}{N}}(1-\lambda_2)^{\frac{A}{N}}}{\lambda_2^2-\lambda_2+1}\right)\cdot \left(\frac{\lambda_2^2-\lambda_2+1}{\lambda_2^{\frac{N-A}{N}}(1-\lambda_2)^{\frac{A}{N}}}\right). 
\end{aligned}
\end{equation*}
For $\rho = (\rho_1,\rho_2)\in G^2$, $\phi_1(\rho)$ (resp. $\phi_2(\rho)$) is determined by $\rho_1$ (resp. $\rho_2$), so we denote $\phi_1(\rho_1)$ (resp. $\phi_2(\rho_2)$).
Then we can check that these 1-cocycles satisfies the following relations
\begin{equation*}
\begin{aligned}
&\eta_1(\underline{\rho}_1) = \phi_1(\rho_1)^N\cdot (\mathrm{sgn}(\underline{\rho}_1))^{N-A}, \\
&\eta_2(\underline{\rho}_2) = \phi_2(\rho_2)^N\cdot (\mathrm{sgn}(\underline{\rho}_2))^{A}. \\
\end{aligned}
\end{equation*}

We define the group $\widetilde{G^2}$ as the fiber product $G^2 \times_{\mu_2} \mu_{2N}$ where the group homomorphisms $G^2 \rightarrow \mu_2$ and $\mu_{2N}\rightarrow \mu_2$ is given by
\begin{equation*}
\begin{aligned}
&G^2 \rightarrow \mu_2; && (\rho_1,\rho_2) \mapsto \mathrm{sgn}(\underline{\rho}_1)^{N-A}\mathrm{sgn}(\underline{\rho}_2)^{A}, \\
&\mu_{2N} \rightarrow \mu_2; && \zeta \mapsto \zeta^{N}.
\end{aligned}
\end{equation*}
An element of $\widetilde{G^2}$ can be described by a triplet
\begin{equation*}
(\rho_1,\rho_2,\zeta) \in G\times G \times \mu_{2N}
\end{equation*}
satisfying the relation $\zeta^N =  \mathrm{sgn}(\underline{\rho}_1)^{N-A}\mathrm{sgn}(\underline{\rho}_2)^{A}$.

The group $\widetilde{G^2}$ acts on $\P^1\times \P^1\times T$ through $G^2$.
For $\rho = (\rho_1,\rho_2,\zeta)\in \widetilde{G^2}$, we define 
\begin{equation*}
\chi(\rho) = \zeta\cdot \phi_1(\rho)\cdot \phi_2(\rho) \quad \in B^\times.
\end{equation*}
Then $\chi:\widetilde{G^2} \rightarrow B^\times$ is a 1-cocycle satisfying the relation
\begin{equation*}
\chi(\rho)^N = \eta(\rho).
\end{equation*}
Thus by Proposition \ref{liftprop}, $\widetilde{G^2}$-action on $\P^1\times \P^1\times T$ lifts to $\overline{S}$, and by the universality of the blowing up, we have $\widetilde{G^2}$-action on $S$.

\subsection{An explicit description}
For $\rho = (\rho_1,\rho_2,\zeta)\in \widetilde{G^2}$, $\rho$ acts on the local coordinate $v$ on $S$ as follows.
\begin{equation}\label{vexplicit}
\rho^{\sharp}(v)= \left\{
\begin{aligned}
  & \zeta^k\mathrm{sgn}(\underline{\rho_2})\theta_1(\rho)\theta_2(\rho)v && (N-A:\text{odd}), \\
  & \zeta^l\theta_1(\rho)\theta_2(\rho)v && (N-A:\text{even}),\\
\end{aligned}
\right.
\end{equation}
where $k,l$ are integers satisfying $k(N-A)\equiv -1 \pmod{2N}$ and $lA \equiv 1\pmod{2N}$, and $\theta_1,\theta_2$ are coboundary 1-cocycles defined by
\begin{equation*}
\begin{aligned}
&\theta_1(\rho_1)=\rho^\sharp\left(\frac{\sqrt[N]{\lambda_1}}{\sqrt[N]{1-\lambda_1}}\right)\cdot \frac{\sqrt[N]{1-\lambda_1}}{\sqrt[N]{\lambda_1}},\\
&\theta_2(\rho_2)=\rho^\sharp\left(\frac{\sqrt[N]{1-\lambda_2}}{\sqrt[N]{\lambda_2}}\right)\cdot \frac{\sqrt[N]{\lambda_2}}{\sqrt[N]{1-\lambda_2}}. \\
\end{aligned}
\end{equation*}
\begin{proof}

By the construction of the $\widetilde{G^2}$-action on $\overline{S}$, the local coordinate $z$ on $\overline{S}$ transforms
\begin{equation*}
\rho^\sharp(z) = \chi(\rho)^{-1}\frac{\del}{\del x}(\rho^\sharp(x))\frac{\del}{\del y}(\rho^\sharp(y)) z.
\end{equation*}
Since $v^{N-A}f(x) = z$, we have
\begin{equation*}
v^{N-A} \mapsto \zeta^{-1}\phi_1(\rho)^{-1}\phi_2(\rho)^{-1}\frac{\del}{\del x}(\rho^\sharp(x))\frac{\del}{\del y}(\rho^\sharp(y))\frac{f(x)}{\rho_1^\sharp(f(x))}v^{N-A}.
\end{equation*}
Thus it is enough to show that 
\begin{equation*}
\left(\zeta^k\mathrm{sgn}(\underline{\rho_2})\theta_1(\rho)\theta_2(\rho)\right)^{N-A} = \zeta^{-1}\phi_1(\rho)^{-1}\phi_2(\rho)^{-1}\frac{\del}{\del x}(\rho^\sharp(x))\frac{\del}{\del y}(\rho^\sharp(y))\frac{f(x)}{\rho^\sharp(f(x))}
\end{equation*}
for odd $N-A$, and 
\begin{equation*}
\left(\zeta^l\theta_1(\rho_1)\theta_2(\rho_2)\right)^{N-A}=
\zeta^{-1}\phi_1(\rho)^{-1}\phi_2(\rho)^{-1}\frac{\del}{\del x}(\rho^\sharp(x))\frac{\del}{\del y}(\rho^\sharp(y))\frac{f(x)}{\rho^\sharp(f(x))}
\end{equation*}
for even $N-A$.
This can be checked directly.
\end{proof}

\subsection{The group action on normal functions}
In this section, we will describe the $\widetilde{G}^2$-actions on the sheaf $\mathcal{Q}(\omega)$ and the differential operator $\mathscr{D}$.

\begin{prop} \label{Psidef}  \ 
\begin{enumerate}
\item 
For $\rho=(\rho_1, \rho_2, \zeta) \in \widetilde{G^2}$, let $\Psi_{\rho} \colon \rho^* \mathscr{O}_T^{\rm an} \to \mathscr{O}_T^{\rm an}$ be the $\mathscr{O}_T^{\rm an}$-module isomorphism defined by 
$$\varphi \mapsto \chi(\rho)^{-1} \cdot \rho^{\sharp}(\varphi). $$
Then $(\Psi_{\rho})_{\rho \in \widetilde{G^2}}$ defines a $\widetilde{G^2}$-linearization on $\mathscr{O}_T^{\rm an}$. 
Furthermore, $(\Psi_{\rho})_{\rho \in \widetilde{G^2}}$ induces a $\widetilde{G^2}$-linearization on $\mathcal{Q}(\omega)$. 
\item 
Let $\xi$ be an algebraic family of higher Chow cycles on $S \to T$. 
For any $\rho \in \widetilde{G^2}$, we have 
\begin{align*}
\langle \nu_{\rm tr}(\rho^* \xi),\omega\rangle= \Psi_{\rho} (\langle\nu_{\rm tr}(\xi),\omega\rangle). 
\end{align*}
\end{enumerate}
\end{prop}

\begin{proof}
Let $\rho \in \widetilde{G^2}$. 
By the local description of the $\widetilde{G^2}$-action, we can check that $\rho^*\omega = \chi(\rho)\cdot \omega$.
Pulling-back this relation at $\rho(t)\in T$, we have 
\begin{equation}\label{rhotransformula}
(\rho_t^{-1})^*(\omega_{t}) = \chi(\rho^{-1})(\rho(t))\cdot \omega_{\rho(t)}
\end{equation}
Here $\rho_t\colon S_t\rightarrow S_{\rho(t)}$ is the restriction of 
the isomorphism $\rho\colon S\to S$ to fibers at $t,\rho(t)\in T$, and $\chi(\rho^{-1})(\rho(t))\in \C$ is the value of the rational function $\chi(\rho^{-1})$ at $\rho(t)$.
From the equation (\ref{rhotransformula}), for any 2-chain $\Gamma_{\rho(t)}$ on $S_{\rho(t)}$, we have 
\begin{equation}\label{sekibunprop}
\begin{aligned}
\int_{\rho_t^{-1}(\Gamma_{\rho(t)})}\omega_t &= \int_{\Gamma_{\rho(t)}}(\rho^{-1}_t)^*(\omega_t) = \chi(\rho^{-1})(\rho(t))\cdot \int_{\Gamma_{\rho(t)}}\omega_{\rho(t)}  \\
&= \rho^\sharp\left(\chi(\rho^{-1})\right)(t)\cdot \int_{\Gamma_{\rho(t)}}\omega_{\rho(t)} = \chi(\rho)(t)^{-1}\cdot \int_{\Gamma_{\rho(t)}}\omega_{\rho(t)}, 
\end{aligned}
\end{equation}
where the last equality follows from the property of 1-cocycles.

First, we prove (i).
Since $\chi^{-1}$ is a 1-cocycle for the $\widetilde{G^2}$-action on $\mathscr{O}_{T}^{\mathrm{an}}$, it is clear that $(\Psi_\rho)_{\rho\in \widetilde{G}^2}$ defines a $\widetilde{G^2}$-linearization on $\mathscr{O}_{T}^{\mathrm{an}}$.
To see that $(\Psi_\rho)_{\rho\in \widetilde{G}^2}$ induces a $\widetilde{G^2}$-linearization on $\mathcal{Q}(\omega)$,
we will show
$$\Psi_{\rho}(\rho^*\mathcal{P}(\omega)) = \mathcal{P}(\omega). $$ 
Let $U'$ be an open subset of $T$ in the classical topology and $\varphi\in \rho^*\mathcal{P}(\omega)(U')=\mathcal{P}(\omega)(\rho(U'))$.
Then there exists a $C^\infty$-family of 2-cycles $\{\Gamma_t\}_{t\in \rho(U)}$ such that $\varphi(t) = \int_{\Gamma_t}\omega_t$.
For any $t\in U'$, we have 
\begin{equation*}
(\Psi_{\rho}(\varphi))(t) = \chi(\rho)(t)^{-1}\cdot \varphi(\rho(t)) = \chi(\rho)(t)^{-1}\cdot \int_{\Gamma_{\rho(t)}}\omega_{\rho(t)} \underset{\eqref{sekibunprop}}{=} \int_{\rho_t^{-1}(\Gamma_{\rho(t)})}\omega_t\in \mathcal{P}(\omega_t).
\end{equation*}
Since we have $\Psi_{\rho}(\varphi)\in \mathcal{P}(\omega)(U')$, the morphism $\Psi_{\rho}\colon \rho^*\mathscr{O}_T^{\mathrm{an}}\to \mathscr{O}_T^{\mathrm{an}}$ induces a morphism 
\begin{equation} \label{Pisom}
\Psi_{\rho}\colon \rho^*\mathcal{P}(\omega)\to \mathcal{P}(\omega).
\end{equation} 
These morphisms satisfy the relation $\Psi_{gh} = h^*(\Psi_{g})\circ \Psi_h$ for any $g,h\in \widetilde{G^2}$ and $\Psi_{\mathrm{id}}=\mathrm{id}$.
In particular, $\rho^*(\Psi_{\rho^{-1}})$ is the inverse of $\Psi_{\rho}$, thus \eqref{Pisom} is an isomorphism.
Therefore $(\Psi_{\rho})_{\rho\in \widetilde{G^2}}$ induces the $\widetilde{G^2}$-linearization on $\mathcal{Q}(\omega) = \mathscr{O}_T^{\mathrm{an}}/\mathcal{P}(\omega)$.

Secondly, we prove (ii).
It is enough to show that 
\begin{equation}\label{ETS}
\langle \mathrm{ev}_t\left(\nu_{\mathrm{tr}}(\rho^*\xi)\right), [\omega_t]\rangle = \Psi_\rho
(\langle \nu_{\mathrm{tr}}(\xi)), [\omega]\rangle)(t)
\end{equation}
for any $t\in T$.
By Proposition \ref{h1zero}, there exists a positive integer $m$ such that $m$ times the 1-cycle associated with $\xi_{t}\:\:(t\in T)$
is bounded by a 2-chain.
We take a neighborhood $U'$ of $\rho(t)$ and a $C^\infty$-family of 2-chains $\{\Gamma_{t'}\}_{t'\in U'}$ such that $\Gamma_{t'}$ is a 2-chain whose boundary coincides with $m$ times the 1-cycle associated with $\xi_{t'}$.

The right-hand side of (\ref{ETS}) is represented by the value of local holomorphic function 
\begin{equation*}
\rho^{-1}(U')\ni t' \longmapsto \chi(\rho)(t')^{-1}\cdot \frac{1}{m}\int_{\Gamma_{\rho(t')}}\omega_{\rho(t')} \underset{\eqref{sekibunprop}}{=}\frac{1}{m}\int_{\rho_{t'}^{-1}\left(\Gamma_{\rho(t')}\right)}\omega_{t'}\in \C
\end{equation*}
at $t$.
Thus, the right-hand side of \eqref{ETS} is 
\begin{equation}\label{ans}
\Psi_\rho
(\langle \nu_{\mathrm{tr}}(\xi)), [\omega]\rangle)(t)=\frac{1}{m}\displaystyle \int_{\rho_t^{-1}(\Gamma_{\rho(t)})}\omega_{t} \mod \mathcal{P}(\omega_{t}).
\end{equation}
On the other hand, the left-hand side of (\ref{ETS}) is $\langle r((\rho^*\xi)_t), [\omega_t]\rangle  = \langle r(\rho_t^*(\xi_{\rho(t)})), [\omega_t]\rangle $.
Since $\rho_t^{-1}(\Gamma_{\rho(t)})$ is a 2-chain whose boundary coincides with $m$ times the 1-cycle associated with $\rho_t^*(\xi_{\rho(t)})$, the left-hand side of (\ref{ETS}) also coincides with (\ref{ans}).
Hence we have the result.
\end{proof}

\begin{prop}\label{thetadef}
Let $\d_i$ $(i=1, 2)$ be the $1$-cocycle defined by 
\begin{align*}
\delta_i(\rho_i)=\rho_i^{\sharp}\left(\dfrac{\l_i(1-\l_i)}{(\l_i^2-\l_i+1)^2} \right) \cdot \left(\dfrac{(\l_i^2-\l_i+1)^2}{\l_i(1-\l_i)} \right).  
\end{align*}
For $\rho=(\rho_1, \rho_2, \zeta) \in \widetilde{G^2}$, let $\Theta_{\rho} \colon \rho^*(\mathscr{O}^{\rm an}_T)^{\oplus 2} \to (\mathscr{O}^{\rm an}_T)^{\oplus 2}$ be the morphism defined by 
$$\begin{pmatrix}
\varphi_1 \\
\varphi_2 
\end{pmatrix} \mapsto 
\begin{pmatrix}
\chi(\rho)^{-1} \cdot \d_1(\rho_1)^{-1} \cdot \rho^{\sharp} (\varphi_1) \\
\chi(\rho)^{-1} \cdot \d_2(\rho_2)^{-1} \cdot \rho^{\sharp} (\varphi_2) 
\end{pmatrix}. $$
Let $\xi$ be an algebraic family of higher Chow cycles on $S \to T$. 
Then we have
$$\mathscr{D}(\langle\nu_{\rm tr}(\rho^* \xi),\omega\rangle)=\Theta_{\rho}(\mathscr{D}(\langle\nu_{\rm tr}(\xi),\omega\rangle)), $$
where $\mathscr{D}$ is the Picard-Fuchs differential operator defined in Proposition \ref{PF}. 
\end{prop}

\begin{proof}
By the direct computation, we have 
$$\mathscr{D}_{\l_1} \cdot \chi(\rho)^{-1} = \chi(\rho)^{-1} \cdot \d_1(\rho_1)^{-1} \cdot \mathscr{D}^{\rho}_{\l_1}. $$
Here, $\mathscr{D}^{\rho}_{\l_1}$ is the differential operator defined by 
$$\mathscr{D}^{\rho}_{\l_1}=\l_1'(1-\l_1') \dfrac{\partial^2}{(\partial \l_1')^2} + 2\left(1-\frac{A}{N} - \l_1' \right) \dfrac{\partial}{\partial \l_1'} -\frac{A}{N}\left(1-\frac{A}{N}\right),
$$
where $\l_1'=\rho^{\sharp}(\l_1)$.
By definition, for a local section $\varphi$ of $\mathscr{O}_T^{\rm an}$, $\mathscr{D}^{\rho}_{\lambda_1}(\rho^\sharp(\varphi)) = \rho^\sharp(\mathscr{D}^{\rho}_{\lambda_1}(\varphi))$.
Thus, we have 
\begin{align*}
\mathscr{D}_{\l_1} (\Psi_{\rho}(\varphi)) &= \mathscr{D}_{\l_1} \left(\chi(\rho)^{-1} \cdot \rho^{\sharp}(\varphi) \right) = \chi(\rho)^{-1} \cdot \d_1(\rho_1)^{-1} \cdot \mathscr{D}_{\l_1} (\rho^{\sharp}(\varphi)) \\
&=\chi(\rho)^{-1} \cdot \d_1(\rho_1)^{-1} \cdot \rho^{\sharp} (\mathscr{D}_{\l_1}(\varphi)). 
\end{align*}
Similarly, we can prove that 
\begin{align*}
\mathscr{D}_{\l_2} (\Psi_{\rho}(\varphi))&=\chi(\rho)^{-1} \cdot \d_2(\rho_2)^{-1} \cdot \rho^{\sharp} (\mathscr{D}_{\l_2}(\varphi)). 
\end{align*}
In particular, we have $\mathscr{D}(\Psi_{\rho}(\varphi))= \Theta_{\rho}(\mathscr{D}(\varphi))$. 
Then, by Proposition \ref{Psidef}, we have 
$$\mathscr{D}(\langle\nu_{\rm tr} (\rho^* \xi),\omega\rangle) = \mathscr{D}(\Psi_{\rho}(\langle\nu_{\rm tr}(\xi),\omega\rangle)) = \Theta_{\rho} (\mathscr{D}(\langle\nu_{\rm tr}(\xi),\omega\rangle)). $$
\end{proof}

\section{Proof of the main theorem}\label{main}
In this section, we prove the main theorem.
First, we describe the $\widetilde{G^2}$-action on the images of normal functions of our families of cycles $\xi_i^{(j)}$ under the differential operators $\mathscr{D}_{\l_1},\mathscr{D}_{\l_2}$, explicitly.

\begin{prop}\label{mainprop}
For $i\in \Z/N\Z$,the images of $\nu_{\rm tr}(\xi_0^{(i)})$ and $\nu_{\rm tr}(\xi_1^{(i)})$ under the Picard-Fuchs differential operator $\mathscr{D}$ are as follows: 
\begin{align*}
&\mathscr{D}(\langle\nu_{\rm tr}(\xi_0^{(i)}),\omega\rangle)=\frac{(1-\zeta^A)\zeta^{Ai}}{\l_1-\l_2}
\begin{pmatrix}
1 \\
-1
\end{pmatrix}, \\
&\mathscr{D}(\langle\nu_{\rm tr}(\xi_1^{(i)}),\omega\rangle)=\frac{(1-\zeta^A)\zeta^{Ai}}{\l_1-\l_2}
\begin{pmatrix}
\dfrac{(1-\l_2)^{\frac{A}{N}}}{(1-\l_1)^{\frac{A}{N}}}\\
-\dfrac{(1-\l_2)^{\frac{N-A}{N}}}{(1-\l_1)^{\frac{N-A}{N}}}
\end{pmatrix}. 
\end{align*}
\end{prop}

\begin{proof}
By Lemma \ref{covering}, it is enough to show in the case $i=0$.
Let $\rho=(\rho_1, {\rm id}, 1) \in \widetilde{G^2}$, where $\rho_1\in \Aut(R/R_0)$ is defined by
$\rho_1(\sqrt[N]{1-\l_1})=\zeta \sqrt[N]{1-\l_1}$ and $\rho_1(\sqrt[N]{\l_1}) = \sqrt[N]{\l_1}$. 
Let $\rho^i=\underbrace{\rho\circ \cdots \circ \rho}_i$.
Then we have 
\begin{align*}
& \d_1(\rho_1^i)= \d_2(\mathrm{id})=1, \\
&\chi(\rho^i)= 1 \cdot \phi_1(\rho_1^i) \cdot \phi_2({\rm id})=\zeta^{i(N-A)}. 
\end{align*}
By the local description of the $\widetilde{G^2}$-action in \eqref{vexplicit}, we see that 
\begin{equation}\label{rhotransformula2}
(\rho^i)^*(\xi_1^{(0)})=\xi_1^{(0)}, \quad (\rho^i)^*(\xi_0^{(0)})=\xi_0^{(i)}. 
\end{equation}
By the definition of $\xi_0^{(j)}$ and $\gcd(N,A)=1$, we have 
$$\sum_{i=0}^{N-1} (\rho^i)^*(\xi_0^{(0)})=\sum_{i=0}^{N-1} \xi_0^{(i)}=0. $$
On the other hand, by Theorem \ref{main:1} and Proposition \ref{thetadef}, we have 
\begin{align*}
\mathscr{D}(\langle \nu_{\rm tr}((\rho^i)^*(\xi_1^{(0)} -\xi_0^{(0)})),\omega\rangle) 
&=\Theta_{\rho^i} (\mathscr{D}(\langle \nu_{\rm tr} (\xi_1^{(0)} - \xi_0^{(0)}),\omega\rangle) \\
&=\dfrac{1-\zeta^A}{\l_1-\l_2} 
\begin{pmatrix}
-\zeta^{Ai} + \dfrac{(1-\l_2)^{\frac{A}{N}}}{(1-\l_1)^{\frac{A}{N}}} \\
\zeta^{Ai}- \dfrac{(1-\l_1)^{\frac{N-A}{N}}}{(1-\l_2)^{\frac{N-A}{N}}} \\
\end{pmatrix}. 
\end{align*}
Since $\gcd(N, A)=1$, we have 
\begin{align*}
\sum_{i=0}^{N-1} \mathscr{D}(\langle\nu_{\rm tr} ((\rho^i)^*(\xi_1^{(0)}-\xi_0^{(0)}),\omega\rangle)
= \dfrac{1-\zeta^A}{\l_1-\l_2} 
\begin{pmatrix}
N \cdot \dfrac{(1-\l_2)^{\frac{A}{N}}}{(1-\l_1)^{\frac{A}{N}}} \\
-N \cdot\dfrac{(1-\l_1)^{\frac{N-A}{N}}}{(1-\l_2)^{\frac{N-A}{N}}} \\
\end{pmatrix}
. 
\end{align*}
By the relation \eqref{rhotransformula2}, we have
\begin{align*}
\mathscr{D}(\langle \nu_{\rm tr}((\rho^i)^*(\xi_1^{(0)} -\xi_0^{(0)})),\omega\rangle) &= \mathscr{D}(\langle \nu_{\rm tr}((\rho^i)^*(\xi_1^{(0)} )),\omega\rangle)-\mathscr{D}(\langle \nu_{\rm tr}((\rho^i)^*(\xi_0^{(0)} )),\omega\rangle)\\
&=\mathscr{D}(\langle\nu_{\mathrm{tr}}(\xi_1^{(0)}),\omega\rangle)-\mathscr{D}(\langle\nu_{\mathrm{tr}}(\xi_0^{(i)}),\omega\rangle).
\end{align*}
Therefore, we conclude that 
\begin{align*}
&\mathscr{D} (\langle \nu_{\rm tr}(\xi_1^{(0)}), \omega \rangle)=\dfrac{1-\zeta^A}{\l_1-\l_2}
\begin{pmatrix}
 \dfrac{(1-\l_2)^{\frac{A}{N}}}{(1-\l_1)^{\frac{A}{N}}} \\
- \dfrac{(1-\l_1)^{\frac{N-A}{N}}}{(1-\l_2)^{\frac{N-A}{N}}} \\
\end{pmatrix}, \\
&\mathscr{D} (\langle \nu_{\rm tr}(\xi_0^{(0)}), \omega \rangle)=\dfrac{1-\zeta^A}{\l_1-\l_2}
\begin{pmatrix}
1 \\
-1 \\
\end{pmatrix}
, 
\end{align*}
which finishes the proof. 
\end{proof}

We define the $\Q$-linear subspace $\Xi^{\rm can}$ of $\Gamma(T, \mathcal{Q})$ by 
\begin{align*}
& \Xi^{\rm can} = \left\langle \nu_{\rm tr}(\xi_0^{(i)}), \nu_{\rm tr}(\xi_1^{(i)}) \mid i \in \Z/N\Z \ \right\rangle_{\Q}. 
\end{align*}
By Proposition \ref{mainprop}, we have 
$$\dim_{\Q} \Xi^{\rm can}= 2 \cdot \varphi(N). $$

Let $\Delta_0\subset G_0^2$ be the subgroup $\{(\underline{\rho},\underline{\rho})\in G_0^2\mid \underline{\rho}\in G_0\}$ and $\Delta$ be the inverse image of $\Delta_0$ by $\widetilde{G^2}\to G^2\to G_0^2$.
Similarly, let $I_0 = \langle((0 \ 1),(0\ 1))\rangle \subset\Delta_0$ and $I\subset \widetilde{G}^2$ be the inverse image of $I_0$.
We define the $\Q$-linear subspace $\Xi_\Delta$ and $\Xi$ by
$$\Xi_{\Delta}= \Q[\Delta] \cdot \Xi^{\rm can}, \quad \Xi= \Q[\widetilde{G^2}] \cdot \Xi^{\rm can}, $$
i.e. the $\Q$-linear subspace of $\Gamma(T,\mathcal{Q})$ generated by the images of $\Xi^{\mathrm{can}}$ under the $\Delta$ and $\widetilde{G^2}$ action, respectively.

\begin{thm}
Suppose that $N\neq 2$. Then we have 
$$\dim_\Q \Xi_\Delta=6 \cdot \varphi(N). $$
\end{thm}
\begin{proof}
Since the $I_0$-action on $\P^1\times \P^1\times T$ stabilizes the family of diagonal curves $D$, the $I$-action on $S$ stabilizes the family of curves $Z$.
Furthermore, the $I$-action stabilizes the set $\{Q_{(0,0)}, Q_{(1,1)}\}$ of families of curves.
Thus, for geometric reason, the group $I$ stabilizes $\Xi^{\mathrm{can}}$.
Hence we have 
$$\dim_\Q \Xi_\Delta\le |\Delta/I|\cdot \dim_\Q \Xi^{\mathrm{can}} = |\Delta_0/I_0|\cdot 2\cdot \varphi(N)= 6\cdot \varphi(N).$$
We will show that the inverse direction of the inequality.
The representative of $\Delta_0/I_0$ is given by 
$$(\mathrm{id},\mathrm{id}),\quad ((0 \ 1/\l), (0 \ 1/\l)), \quad ((1 \ 1/\l), (1 \ 1/\l)). $$
Let $\tau\in G$ be the automorphism defined by
$$
\tau^\sharp\left(\sqrt[N]{\lambda}\right) = \sqrt[N]{1-\lambda},\quad \tau^\sharp\left(\sqrt[N]{1-\lambda}\right) = \sqrt[N]{\lambda}.
$$
Then $\rho=(\tau,\tau,-1)\in\widetilde{G^2}$ is a lift of $((0 \ 1/\l), (0 \ 1/\l))$.
We calculate $\mathscr{D}(\langle \nu_{\mathrm{tr}}(\rho^*\xi_0^{(0)}),\omega\rangle)$ and $\mathscr{D}(\langle \nu_{\mathrm{tr}}(\rho^*\xi_1^{(0)}),\omega\rangle)$.
Since we have 
$$
\chi(\rho) = \lambda_1^{\frac{N-2A}{N}}(1-\lambda_1)^{\frac{2A-N}{N}} \lambda_2^{\frac{2A-N}{N}}(1-\lambda_2)^{\frac{N-2A}{N}},\quad \delta_1(\tau)=\delta_2(\tau)=1,
$$
by Proposition \ref{thetadef}, 
\begin{align*}
&\mathscr{D}_{\l_1}(\langle \nu_{\mathrm{tr}}(\rho^*\nu(\xi_0^{(0)})),\omega\rangle)
=-\frac{1-\zeta^A}{\l_1-\l_2}\l_1^{\frac{2A-N}{N}}(1-\lambda_1)^{\frac{N-2A}{N}}\l_2^{\frac{N-2A}{N}}(1-\lambda_2)^{\frac{2A-N}{N}}, \\
&\mathscr{D}_{\l_1}(\langle \nu_{\mathrm{tr}}(\rho^*\nu(\xi_1^{(0)})),\omega\rangle)
=-\frac{1-\zeta^A}{\l_1-\l_2}\l_1^{\frac{A-N}{N}}(1-\lambda_1)^{\frac{N-2A}{N}}\l_2^{\frac{N-A}{N}}(1-\lambda_2)^{\frac{2A-N}{N}}.
\end{align*}
Similarly, let $\tau'\in G$ be the automorphism defined by
$$
(\tau')^\sharp\left(\sqrt[N]{\l}\right) = \frac{1}{\sqrt[N]{\l}},\quad (\tau')^\sharp\left(\sqrt[N]{1-\l}\right) = \zeta'\cdot \frac{\sqrt[N]{1-\l}}{\sqrt[N]{\l}}
$$
where $\zeta'$ is a $N$th root of $-1$.
Then $\rho' = (\tau',\tau',-1)\in\widetilde{G^2}$ is a lift of 
$((1 \ 1/\l), (1 \ 1/\l))$.
Since we have 
$$
\chi(\rho') = \l_1^{\frac{N-A}{N}}\l_2^{\frac{A}{N}},\quad \delta_1(\tau') =-\l_1,\quad \delta_2(\tau')=-\l_2,
$$
by Proposition \ref{thetadef}, 
\begin{align*}
&\mathscr{D}_{\l_1}(\langle \nu_{\mathrm{tr}}((\rho')^*\nu(\xi_0^{(0)})),\omega\rangle)
=\frac{1-\zeta^A}{\l_1-\l_2}\l_1^{\frac{A-N}{N}}\l_2^{\frac{N-A}{N}}, \\
&\mathscr{D}_{\l_1}(\langle \nu_{\mathrm{tr}}((\rho')^*\nu(\xi_1^{(0)})),\omega\rangle)
=\frac{1-\zeta^A}{\l_1-\l_2}\l_1^{\frac{2A-N}{N}}(1-\lambda_1)^{-\frac{A}{N}}\l_2^{\frac{N-2A}{N}}(1-\lambda_2)^{\frac{A}{N}}.
\end{align*}
We will show that 
\begin{equation}\label{basis}
\begin{aligned}
&\mathscr{D}_{\l_1}(\langle \nu_{\mathrm{tr}}(\nu(\xi_0^{(0)})),\omega\rangle),&&
\mathscr{D}_{\l_1}(\langle \nu_{\mathrm{tr}}(\nu(\xi_1^{(0)})),\omega\rangle),
&&\mathscr{D}_{\l_1}(\langle \nu_{\mathrm{tr}}(\rho^*\nu(\xi_0^{(0)})),\omega\rangle),\\
&\mathscr{D}_{\l_1}(\langle \nu_{\mathrm{tr}}(\rho^*\nu(\xi_1^{(0)})),\omega\rangle),
&&\mathscr{D}_{\l_1}(\langle \nu_{\mathrm{tr}}((\rho')^*\nu(\xi_0^{(0)})),\omega\rangle),
&&\mathscr{D}_{\l_1}(\langle \nu_{\mathrm{tr}}((\rho')^*\nu(\xi_1^{(0)})),\omega\rangle)
\end{aligned}
\end{equation}
are linearly independent over $\Q(\zeta)$.
If not, there exists nontrivial vector $(c_1,\dots,c_6)\in \Q(\zeta)^{\oplus 6}$ such that
\begin{align*}
&c_1+c_2(1-\l_1)^{-\frac{A}{N}}(1-\l_2)^{\frac{A}{N}} + c_3\l_1^{\frac{2A-N}{N}}(1-\lambda_1)^{\frac{N-2A}{N}}\l_2^{\frac{N-2A}{N}}(1-\lambda_2)^{\frac{2A-N}{N}} \\
&+c_4\l_1^{\frac{A-N}{N}}(1-\lambda_1)^{\frac{N-2A}{N}}\l_2^{\frac{N-A}{N}}(1-\lambda_2)^{\frac{2A-N}{N}} +c_5\l_1^{\frac{A-N}{N}}\l_2^{\frac{N-A}{N}}\\
&+c_6\l_1^{\frac{2A-N}{N}}(1-\lambda_1)^{-\frac{A}{N}}\l_2^{\frac{N-2A}{N}}(1-\lambda_2)^{\frac{A}{N}}=0.
\end{align*}
If we set $X=\left(\frac{\l_2}{\l_1}\right)^\frac{1}{N}$ and $Y=\left(\frac{1-\l_2}{1-\l_1}\right)^{\frac{1}{N}}$, 
the above relation is
$$
c_1 + c_2Y^{A} + c_3X^{N-2A}Y^{2A-N} + c_4X^{N-A}Y^{2A-N} +c_5X^{N-A} + c_6X^{N-2A}Y^{A}=0.
$$
By the assumption, we have $2A-N\neq 0$, and $X$ and $Y$ are $\C$-linearly independent, the above relation implies $c_1=c_2=\cdots=c_6=0$. 
This contradicts to the assumption.
The $\Q$-linear subspace $\Xi_\Delta$ contains the $\zeta^i$ times each member in \eqref{basis} for $i=0,1,\dots, N-1$.
Thus each member in \eqref{basis} generates $\varphi(N)$-dimensional subspace $V_1,V_2,\dots,V_6$.
By the $\Q(\zeta)$-linearly independence of \eqref{basis} 
implies that the natural map $V_1\oplus V_2\oplus\cdots\oplus V_6\to \Xi_\Delta$ is injective.
This proves the inverse of the inequality.
\end{proof}

\begin{rmk}
When $N=2$ (the case considered in \cite{Sato}), \eqref{basis} are not linearly independent over $\Q$, and generate only a 
$3$-dimensional $\Q$-linear subspace. 
\end{rmk}

The above proof shows that for each $\nu\in \Xi_\Delta$, $\mathscr{D}(\langle\nu,\omega\rangle)$ can be expressed by
\begin{equation}\label{normalfunc}
\begin{aligned}
\mathscr{D}_{\l_1}(\langle\nu,\omega\rangle)
= \frac{1}{\l_1-\l_2}(&c_1 + c_2Y^{A} + c_3X^{N-2A}Y^{2A-N}\\
&+ c_4X^{N-A}Y^{2A-N} +c_5X^{N-A} + c_6X^{N-2A}Y^{A}),  
\end{aligned}
\end{equation}
where $c_1,c_2,\dots,c_6\in \Q(\zeta)$. 
We regard the right-hand side of \eqref{normalfunc} as a rational function on $\overline{T}=R\times R$.
Then we have the following.

\begin{lem}\label{polelemma}
For $i,j\in\{0,1,\dots,N-1\}$, let $F_{i,j}$ be the curve on $\overline{T}$ defined by the following equations
$$
\zeta^i\l_1^{\frac{1}{N}}-\l_2^{\frac{1}{N}}=0,\quad (1-\l_2)^{\frac{1}{N}}-\zeta^j(1-\l_1)^{\frac{1}{N}}=0. 
$$
Then the right-hand side of \eqref{normalfunc} has a pole along at least one $F_{i,j}$ if $\nu\neq 0$.
\end{lem}
\begin{proof}
It is enough to show that if the right-hand side of (\ref{normalfunc}) has no poles along all $F_{i,j}$, then we have $c_1=c_2=\cdots =c_6=0$.
Note that $\l_1-\l_2$ has a pole along each $F_{i,j}$.
Therefore, if (\ref{normalfunc}) has no poles along all $F_{i,j}$, 
$$c_1 + c_2Y^{A} + c_3X^{N-2A}Y^{2A-N}+ c_4X^{N-A}Y^{2A-N} +c_5X^{N-A} + c_6X^{N-2A}Y^{A}
$$
has zero along all $F_{i,j}$.
On $F_{i,j}$, we have $X=\zeta^i$ and $Y=\zeta^j$, hence this implies
$$c_1 + c_2\zeta^{Aj} + c_3\zeta^{(N-2A)i}\zeta^{(2A-N)j}+ c_4\zeta^{(N-A)i}\zeta^{(2A-N)j} +c_5\zeta^{(N-A)i} + c_6\zeta^{(N-2A)i}\zeta^{Aj}=0
$$
for any $i,j$, i.e. 
$$(c_1 + c_2\zeta^{Aj}) + (c_3\zeta^{(2A-N)j}+c_6\zeta^{Aj})
\zeta^{(N-2A)i}+ (c_4\zeta^{(2A-N)j}+c_5)\zeta^{(N-A)i}=0. 
$$
By the assumption \eqref{gcdassumption}, if $i$ runs over $0,1,\ldots, N-1$, this relation implies 
$$c_1 + c_2\zeta^{Aj}=c_3\zeta^{(2A-N)j}+c_6\zeta^{Aj}=c_4\zeta^{(2A-N)j}+c_5=0.$$
If $j$ runs over $0,1,\ldots, N-1$, these relations implies
$c_1=c_2=\cdots =c_6=0$.
Thus we have the result.
\end{proof}

\begin{thm} \label{mainthm}
Suppose that $N\neq 2$.
We have
$$\dim_\Q \Xi \ge 36 \cdot \varphi(N). $$
In particular, we have $\rk \mathrm{CH}^2(S_t,1)_{\mathrm{ind}}\ge 36\cdot \varphi(N)$ for very general $t\in T$.
\end{thm}

\begin{proof}
Let $\rho_1,\rho_2,\ldots,\rho_6\in G_0^2$ be the elements given by 
$$(\mathrm{id},\mathrm{id}), \ (\mathrm{id}, (0\ 1)), \ (\mathrm{id}, \ (0\ 1/\l)), \ (\mathrm{id}, (1\ 1/\l)), \ (\mathrm{id}, (0\ 1 \ 1/\l))\text{ and } \ (\mathrm{id}, (0\ 1/\l \ 1)), $$
respectively. 
These are representatives of $G_0^2/\Delta_0$.
We take their lifts $\widetilde{\rho}_1,\widetilde{\rho}_2, \ldots, \widetilde{\rho}_6\in \widetilde{G^2}$ by $\widetilde{G^2}\to G_0^2$.
We will show that the natural map
\begin{equation}\label{injective}
\bigoplus_{i=1}^6\widetilde{\rho}_i^*\Xi_\Delta \to \Xi
\end{equation}
is injective.
For $i=1,\ldots,6$, let $\nu_i\in \Xi_\Delta$.
We will show that if $\nu_i$ satisfies 
$$\widetilde{\rho}_1^*\nu_1+\widetilde{\rho}_2^*\nu_2+\cdots +\widetilde{\rho}_6^*\nu_6=0,
$$
then $\nu_1=\nu_2=\cdots =\nu_6=0$.

We assume $\widetilde{\rho}_1^*\nu_1+\widetilde{\rho}_2^*\nu_2+\cdots +\widetilde{\rho}_6^*\nu_6=0$, then we have
\begin{equation}\label{Drelation}
    \mathscr{D}_{\l_1}(\langle\widetilde{\rho}_1^*\nu_1,\omega\rangle)=   
    -\mathscr{D}_{\l_1}(\langle\widetilde{\rho}_2^*\nu_2,\omega\rangle)-\cdots - \mathscr{D}_{\l_1}(\langle\widetilde{\rho}_6^*\nu_6,\omega\rangle).
\end{equation}
We regard $\mathscr{D}_{\l_1}(\langle\widetilde{\rho}_i^*\nu_i,\omega\rangle)$ as rational functions on $\overline{T}=R\times R$ and compare the poles of the both side of \eqref{Drelation}.

By Proposition \ref{thetadef}, we have $\mathscr{D}_{\l_1}(\langle\widetilde{\rho}_i^*\nu_i,\omega\rangle)=\Theta_{\widetilde{\rho}_i}\left(\mathscr{D}_{\l_1}(\langle\nu_i,\omega\rangle)\right)$.
By Lemma \ref{polelemma}, if $\nu_1\neq 0$, there exists $j_1,k_1\in \{0,1,\ldots,N-1\}$ such that $\mathscr{D}_{\l_1}(\langle\nu_1,\omega\rangle)$ has a pole on $F_{j_1,k_1}$.
Thus, the left-hand side of \eqref{Drelation} has a pole on $\widetilde{\rho}_1^{-1}\left(F_{j_1,k_1}\right)$, whose equation is given by 
$$
\zeta^{j'_1}\l_1^{\frac{1}{N}}-\l_2^{\frac{1}{N}}=0,\quad(1-\l_2)^{\frac{1}{N}}-\zeta^{k_1'}(1-\l_1)^{\frac{1}{N}}=0
$$
for some $j_1',k_1'\in \Z$ depending on the choice of the lift $\widetilde{\rho}_1\in\widetilde{G^2}$.
On the other hand, 
since we have $\delta_1(\widetilde{\rho}_i),\chi(\widetilde{\rho}_i)\in \Gamma(\overline{T},\mathscr{O}_{\overline{T}}^\times)$, by \eqref{normalfunc}, $\mathscr{D}_{\l_1}(\langle\widetilde{\rho}_i^*\nu_i,\omega\rangle)=\Theta_{\widetilde{\rho}_i}\left(\mathscr{D}_{\l_1}(\langle\nu_i,\omega\rangle)\right)$ is regular outside of the zeros of $\widetilde{\rho}_i^\sharp(\l_1-\l_2)$.
For $i=2,3,\ldots,6$, $\widetilde{\rho_i}^\sharp(\l_1-\l_2)$ is 
$$
\l_1-\frac{\l_2}{\l_2-1},\quad \l_1+\l_2-1, \quad \l_1-\frac{1}{\l_2}, \quad  \l_1-\frac{1}{1-\l_2},\quad 
\l_1 -\frac{\l_2-1}{\l_2},
$$
respectively.
Since $\widetilde{\rho}_1^{-1}(F_{j_1,k_1})$ is not contained in zero loci of these functions, each term in the right-hand side of \eqref{Drelation} has no poles on $\widetilde{\rho}_1^{-1}(F_{j_1,k_1})$.
This implies $\nu_1=0$.

Then by \eqref{Drelation}, we have 
$$
\mathscr{D}_{\l_1}(\langle\widetilde{\rho}_2^*\nu_2,\omega\rangle)=   
-\mathscr{D}_{\l_1}(\langle\widetilde{\rho}_3^*\nu_3,\omega\rangle)-\cdots - \mathscr{D}_{\l_1}(\langle\widetilde{\rho}_6^*\nu_6,\omega\rangle).
$$
Continuing the same discussions, 
we can show $\nu_2=0$, $\nu_3=0,\dots, \nu_6=0$, inductively.
Thus we have shown that \eqref{injective} is injective.
Then we have
$$\dim_\Q\Xi\ge 6\cdot \dim_\Q\Xi_\Delta = 36\cdot \varphi(N), $$
which finishes the proof of the former statement.

For the latter statement, consider the subgroup
$$
\widetilde{\Xi}_t = \left\langle\rho^*_t((\xi_0^{(i)})_{\rho(t)}),\rho^*_t((\xi_0^{(i)})_{\rho(t)})\mid i\in \Z/N\Z, \ \rho\in \widetilde{G^2}\right\rangle_\Z\quad \subset \mathrm{CH}^2(S_t,1)
$$
for $t\in T$, where $\rho_t\colon S_t\to S_{\rho(t)}$ is the pull-back of $\rho\colon S\to S$.
By Lemma \ref{basiclem}, $\Xi\rightarrow \mathrm{ev}_t(\Xi)$ is bijective for very general $t\in T$.
By the former statement, we have $\dim_\Q \mathrm{ev}_t(\Xi)\geq  36\cdot \varphi(N)$.
Since $\mathrm{ev}_t(\Xi)$ is generated by $r(\widetilde{\Xi}_t)$ over $\Q$ and the transcendental regulator map $r$ factors $\mathrm{CH}^2(X,1)_{\mathrm{ind}}$ (see Proposition \ref{transregprop}), this shows that $\widetilde{\Xi}_t$ generates a subgroup in $\mathrm{CH}^2(X,1)_{\mathrm{ind}}$ of rank $36 \cdot \varphi(N)$.
\end{proof}

\section*{Acknowledgment}
This article was written while the first author was visiting the University
of Winnipeg. He would like to thank the Department of Mathematics and Statistics for their hospitality. 
The first author is supported by Waseda University Grant for Early Career Researchers (Project number: 2025E-041). 
The second author is supported by JSPS KAKENHI 21H00971.

\end{document}